\documentclass[12pt]{amsart}

\usepackage[english]{babel}
\usepackage[T1]{fontenc}
\usepackage[utf8]{inputenc} 
\usepackage{amssymb,latexsym}
\usepackage{mathtools}
\usepackage{enumerate}  
\usepackage{a4wide}    

\usepackage[colorlinks=true,citecolor=blue,linkcolor=blue]{hyperref}
\usepackage{cleveref}
\usepackage{latexsym,amsmath,amssymb}
\usepackage{accents}
\usepackage{color}  
\title[On the Calderon-Zygmund property]{On the Calderon-Zygmund property of Riesz-transform type operators arising in nonlocal equations}

\author{S. Yeepo}
\address[Sasikarn Yeepo]{Department of Mathematics and Computer Science, Faculty of Science, Chulalongkorn University, Bangkok 10330, Thailand}
\email{6072857023@student.chula.ac.th}

\author{W. Lewkeeratiyutkul}
\address[Wicharn Lewkeeratiyutkul]{Department of Mathematics and Computer Science, Faculty of Science, Chulalongkorn University, Bangkok 10330, Thailand}
\email{wicharn.l@chula.ac.th}

\author{S. Khomrutai}
\address[Sujin Khomrutai]{Department of Mathematics and Computer Science, Faculty of Science, Chulalongkorn University, Bangkok 10330, Thailand}
\email{sujin.k@chula.ac.th}

\author{A. Schikorra}
\address[Armin Schikorra]{Department of Mathematics,
University of Pittsburgh,
301 Thackeray Hall,
Pittsburgh, PA 15260, USA}
\email{armin@pitt.edu}

plus 6pt minus 12pt
\def\eps{\varepsilon}

\newtheorem{theorem}{Theorem}

\newtheorem{lemma}[theorem]{Lemma}

\newtheorem{proposition}[theorem]{Proposition}

\theoremstyle{definition}
\newtheorem{remark}[theorem]{Remark}

\newcommand{\R}{\mathbb{R}}

\newcommand{\brac}[1]{\left (#1 \right )}
\newcommand{\abs}[1]{\left |#1 \right |}

\newcommand{\barint}{
\rule[.036in]{.12in}{.009in}\kern-.16in \displaystyle\int }

\newcommand{\barcal}{\mbox{$ \rule[.036in]{.11in}{.007in}\kern-.128in\int $}}
\def\mvint_#1{\mathchoice
          {\mathop{\vrule width 6pt height 3 pt depth -2.5pt
                  \kern -8pt \intop}\nolimits_{\kern -3pt #1}}%
          {\mathop{\vrule width 5pt height 3 pt depth -2.6pt
                  \kern -6pt \intop}\nolimits_{#1}}%
          {\mathop{\vrule width 5pt height 3 pt depth -2.6pt
                  \kern -6pt \intop}\nolimits_{#1}}%
          {\mathop{\vrule width 5pt height 3 pt depth -2.6pt
                  \kern -6pt \intop}\nolimits_{#1}}}

\numberwithin{theorem}{section} \numberwithin{equation}{section}

\newcommand{\lap}{\Delta }

\newcommand{\aleq}{\lesssim}
\newcommand{\ageq}{\gtrsim}
\newcommand{\aeq}{\approx}
\newcommand{\Rz}{\mathcal{R}}
\newcommand{\laps}[1]{(-\lap) ^{\frac{#1}{2}}}

\newcommand{\lapms}[1]{I^{#1}}

\begin{document}
\begin{abstract}
We show that the operator
\[
 T_{K,s_1,s_2}f(z) := \int_{\R^n} A_{K,s_1,s_2}(z_1,z_2) f(z_2)\, dz_2
\]
is a Calderon-Zygmund operator. Here for $K \in L^\infty(\R^n \times \R^n)$, and $s,s_1,s_2 \in (0,1)$ with $s_1+s_2 = 2s$ we have
\[
  A_{K,s_1,s_2}(z_1,z_2) = \int_{\R^n} \int_{\R^n} \frac{K(x,y) \brac{|x-z_1|^{s_1-n} -|y-z_1|^{s_1-n}}\, \brac{|x-z_2|^{s_2-n} -|y-z_2|^{s_2-n}}}{|x-y|^{n+2s}}\, dx\, dy.
\]
This operator is motivated by the recent work \cite{MSY20} where it appeared as analogue of the Riesz transforms for the equation
\[
 \int_{\R^n} \int_{\R^n} \frac{K(x,y) (u(x)-u(y))\, (\varphi(x)-\varphi(y))}{|x-y|^{n+2s}}\, dx\, dy = f[\varphi].
\]
\end{abstract}

\maketitle
% \tableofcontents

\section{Introduction}

Riesz transforms arise naturally in linear PDEs of divergence form
\begin{equation}\label{eq:divform}
\sum_{i=1}^{n} \partial_i (A_{ij} \partial_i u) = g  \quad \text{in $\R^n$}
\end{equation}
or nondivergence form
\begin{equation}\label{eq:nondivform}
\sum_{i,j=1}^{n} A_{ij} \partial_i (\partial_j u) = g \quad \text{in $\R^n$}.
\end{equation}
They are the ``zero-order''  structural part of the PDE. 
For example, for the divergence form equation \eqref{eq:divform}, 
if we set $f := \laps{1} u$ and apply the Riesz potential $\lapms{1}$ to equation \eqref{eq:divform} (for a definition of these operators, see \Cref{s:L2}), then \eqref{eq:divform} at least formally is equivalent to
\[
\sum_{i=1}^{n} \Rz_i (A_{ij} \Rz_i f) = \lapms{1}g \quad \mbox{in } \R^n.
\]
In particular, if we define the operator
\[
Tf := \sum_{i=1}^{n} \Rz_i (A_{ij} \Rz_i f)
\]
%For the divergence form, its the operator
%\[
% Tf:= \sum_{i=1}^{n} \Rz_i (A \Rz_i f)
%\]
then \eqref{eq:divform} is equivalent to the equation
\[
 T (\laps{1} u) = \lapms{1} g \quad \text{in $\R^n$}.
\]
%For nondivergence form \eqref{eq:nondivform} it's the operator
%\[
% Tf:= \sum_{i=1}^{n} A \Rz_i \Rz_i f = A f.
%\]
%which transforms \eqref{eq:divform} into
%\[
% T(\lap u) = f.
%\]
Existence, regularity, and uniqueness of linear PDE \eqref{eq:divform} or \eqref{eq:nondivform} (or their adaptations with lower order terms) are thus intrinsically related to harmonic analysis results on boundedness and invertibility of the operator $T$. 
The most basic property (which is an almost trivial corollary of the boundedness of the Riesz transforms on any $L^p$-space) is that $T$ is a bounded linear operator from $L^p(\R^n)$ to $L^p(\R^n)$ if  $A$ is bounded, and $T$ has a bounded inverse if $A$ is elliptic and belongs to VMO. Indeed, for example, using purely tools from harmonic analysis, \cite{IS98} showed the Calderon-Zygmund $L^p$-theory for elliptic, bounded and $VMO$-coefficents $A$. 

In this note, we want to study the analogue of the operator $T$ in the setting of a popular nonlocal linear equation, namely
\begin{equation}\label{eq:nonloc}
 \mathcal{L}^s_{K}(u,\varphi) := \int_{\R^n} \int_{\R^n} \frac{K(x,y)\ (u(x)-u(y))\, (\varphi(x)-\varphi(y))}{|x-y|^{n+2s}}\, dx\, dy = \int_{\R^n} f(z)\, \varphi(z)\,dz.
\end{equation}
Here we assume that $K(x,y)=K(y,x)$ is measurable in $x$ and $y$ and ellipticity would mean that there exist $\lambda, \Lambda > 0$ such that $\lambda \le K(x,y) \le \Lambda$ for all $x, y \in \R^n$.

Such equations have been studied by numerous authors, we name just a few \cite{K09,DONG20121166,FKV15,KMS15,gale_ofa485955384,Warma,Nowak19,Fall18,euclideuclid.tunis/1543854680,CK20,Nowak20}.

In \cite{MSY20} two of the authors in a joint work with Mengesha, introduced a natural analogue of the operator $T$ -- associated to \eqref{eq:nonloc}: set $A_{K,s_1,s_2}(z_1,z_2)$ to be the following double integral
\[
 \int_{\R^n} \int_{\R^n} K(x,y) \frac{\brac{|x-z_1|^{s_1-n}-|y-z_1|^{s_1-n}}\, \brac{|x-z_2|^{s_2-n}-|y-z_2|^{s_2-n}}}{|x-y|^{n+2s}}\, dx\, dy
\]
where $0< s_1,s_2 < 1$ with $s_1 + s_2 = 2s$, and set
\[
 T_{K,s_1,s_2} f(z_1) := \int_{\R^n} A_{K,s_1,s_2}(z_1,z_2) f(z_2)\, dz_2, \quad z_1 \in \R^n.
\]
Then one can show (and this was crucially used in \cite{MSY20}) that solutions of \eqref{eq:nonloc} satisfy
%is equivalent to \eqref{eq:nonloc} in the sense that solutions to \eqref{eq:nonloc} (for now formally) satisfy
\[
 T_{K,s_1,s_2} (\laps{s_1} u) = \lapms{s_2} f \quad \text{in $\R^n$}.
\]
Again, existence, uniqueness and regularity to equations \eqref{eq:nonloc} are related to boundedness and invertibility of $T_{K,s_1,s_2}$. 

The main result of this work is to show that $T_{K,s_1,s_2}$ is a \emph{Calderon-Zygmund operator} and is in particular a bounded linear operator from $L^p(\R^n)$ to $L^p(\R^n)$. Namely we have
\begin{theorem}\label{th:cz}
	Let $K \in L^\infty(\R^n\times \R^n)$. Then $T_{K,s_1,s_2}$ is a Calderon-Zygmund operator. In particular, for any $p \in (1,\infty)$, we have
	\begin{enumerate}
		\item $\|T_{K,s_1,s_2} f\|_{L^{1,\infty}(\R^n)} \leq C\, \|K\|_{L^\infty(\R^n\times \R^n)}\, \|f\|_{L^1(\R^n)}$,
		\item $\|T_{K,s_1,s_2} f\|_{L^p(\R^n)} \leq C\, \|K\|_{L^\infty(\R^n\times \R^n)}\, \|f\|_{L^p(\R^n)}$,
		\item $[T_{K,s_1,s_2} f]_{BMO(\R^n)} \leq C\, \|K\|_{L^\infty(\R^n\times \R^n)}\, \|f\|_{L^\infty(\R^n)}$.
	\end{enumerate}
\end{theorem}
Here $L^{1,\infty}(\R^n)$ denotes the \emph{weak $L^1$ space} defined as the set of measurable function $g$ such that
\[
\|g\|_{L^{1,\infty}(\R^n)} := \sup_{\lambda>0} \lambda |\{ x \in \R^n : |g(x)| > \lambda \}| < \infty
\]
and $BMO(\R^n)$ denotes the set of locally integrable functions on $\R^n$ which is of \emph{bounded mean oscillation}, that is, $b \in BMO(\R^n)$ if and only if
\[
[b]_{BMO(\R^n)} := \sup_Q \frac{1}{|Q|}\int_Q \left|b(x) - \frac{1}{|Q|}\int_Q b(y)\,dy \right|\,dx < \infty,
\]	
where the supremum is taken over all cubes $Q$ in $\R^n$. 

We establish \Cref{th:cz} by proving that $T_{K,s_1,s_2}$ is bounded from $L^2$ to itself, see \Cref{la:L2bound} below, and showing that $A_{K,s_1,s_2}$ is a \emph{standard kernel}, cf. \cite[Definition 4.1.2]{GMF}, which is the main contribution of this work. 
\begin{proposition}\label{pr:se:1}
	For any  $z_1 \neq z_2 \in \R^n$ we have $A_{K,s_1,s_2}$ satisfies the size condition 
	\begin{equation}\label{eq:CZ:g1}
	|A_{K,s_1,s_2}(z_1,z_2)| \aleq  \frac{\|K\|_{L^\infty(\R^n\times \R^n)}}{|z_1-z_2|^{n}}  
	\end{equation}
	and for some $\alpha > 0$  the regularity conditions 
	\begin{equation}\label{eq:CZ:g2}
	\left |A_{K,s_1,s_2}(z_1+h,z_2) - A_K(z_1,z_2) \right | \aleq  \frac{|h|^{\alpha}\|K\|_{L^\infty(\R^n\times \R^n)}}{|z_1-z_2|^{n+\alpha}}
	\end{equation}
	whenever $|h| \le \frac{1}{2} |z_1-z_2|$ and 
	\begin{equation}\label{eq:CZ:g3}
	\left |A_{K,s_1,s_2}(z_1,z_2+h) - A_K(z_1,z_2) \right| \aleq  \frac{|h|^{\alpha}\|K\|_{L^\infty(\R^n\times \R^n)}}{|z_1-z_2|^{n+\alpha}}
	\end{equation}
	whenever $|h| \le \frac{1}{2} |z_1-z_2|$.
\end{proposition}
In view of \cite[Theorem~4.2.2, Theorem~4.2.7]{GMF}, \Cref{th:cz} is indeed a consequence of \Cref{la:L2bound} and \Cref{pr:se:1}.

\begin{remark}
An easy extension of \Cref{pr:se:1} is the case where $K=K(z_1,z_2,x,y)$ is bounded. However this does not lead immediately to an extension of \Cref{th:cz}, since the $L^2$-boundedness needs to be shown for that kernel. 
\end{remark}

As an application of our estimate in \Cref{th:cz} we obtain the following regularity results for ``almost constant coefficients'' (but without any further regularity assumption). To our knowledge this is new, although results somewhat similar in spirit (in the context of $L^2$-estimates) have been observed e.g. in \cite[Section 3]{Fall18}. Observe that we obtain this estimate at all differentiability scales below $1$.
\begin{theorem}\label{th:pdeappl}
	For any $s \in (0,1)$ and any $s_1,s_2 \in (0,1)$ with $s_1 + s_2 = 2s$, $s_1 \geq s$ and any $p \in [2,\infty)$ there exists $\eps > 0$ such that the following holds. 
	Let $K: \R^n \times \R^n \to [0,\infty)$ be measurable with $1-\frac{\inf K}{\sup K} < \eps$. Then if $u \in H^{s}(\R^n)$ and $f \in L^2(\R^n)$ solves
	\[
	\mathcal{L}^s_{K} u = \laps{s} f \quad \text{in $\R^n$}
	\]
	then whenever the right-hand side is finite,
	\[
	\|\laps{s_1} u\|_{L^p(\R^n)} \leq C\, \|\laps{s-s_2} f\|_{L^p(\R^n)}. 
	\]
	The smallness constant $\eps > 0$ is uniform in the following sense: If for some $\theta > 0$ we have that $s,s_1,s_2  \in (\theta,1-\theta)$, $p \in [2,\frac{1}{\theta})$ then $\eps$ depends only on $\theta$ and the dimension.
\end{theorem}

{\bf Outline}
The $L^2$-boundedness of $T_{K,s_1,s_2}$ is proved in \Cref{s:L2}. In \Cref{s:SK}, we provide the computations that show that the kernel $A_{K,s_1,s_2}$ is a standard kernel. 
The application \Cref{th:pdeappl}, will be proven in the last section.

{\bf Notation}
$A \aleq B$ means there exists a constant $C > 0$ which is not depending on $A$ and $B$ such that $A \le CB$. $A \aeq B$ means that $A \aleq B$ and $B \aleq A$.

{\bf Acknowledgments} The authors acknowledge funding as follows
\begin{itemize}
	\item AS: Simons foundation, grant no 579261.
	\item SY: Science Achievement Scholarship of Thailand (SAST)
\end{itemize}
The research that lead to this work was carried out while SY was visiting the University of Pittsburgh.

\section{\texorpdfstring{$L^2$}{L2} Boundedness of \texorpdfstring{$T_{K,s_1,s_2}$}{T}}\label{s:L2}

First of all, we give definitions of the operators mentioned in the introduction.
For $1 \le i \le n$, the \textit{Riesz transform} $\Rz_i$ of a function $f$ in the Schwartz class is defined by
\begin{equation}
\mathcal{F}(\Rz_i f)(\xi) = c\, \frac{\xi_i}{|\xi|} \mathcal{F} f(\xi). 
\end{equation}
This operator appears as the derivative $\Rz_i := \partial_i \lapms{1}$ of the Riesz potential $\lapms{1}$.  
Here, $I^s$ denotes the \textit{Riesz potential of order $s$} which is defined by
\[
\mathcal{F}(\lapms{s} f)(\xi) := \frac{1}{c}\, |\xi|^{-s} \mathcal{F} f(\xi).
\]
This operator makes sense (for $f$ a function in the Schwartz class) if $0 \leq s < n$, because $|\xi|^{-s}$ is then locally integrable.
The inverse operator of the Riesz potential is the fractional Laplacian operator $(-\lap)^{\frac{s}{2}}$ for $s \in (0,2)$, i.e.,
$\lapms{s} = (-\lap)^{-\frac{s}{2}}$ and 
\[
\mathcal{F}(\laps{s} f)(\xi) = c\, |\xi|^s \mathcal{F} f(\xi).
\]

Moreover, these operators have a useful integral representation. For a function $f$ in the Schwartz class, the integral form of
the Riesz transform for $i=1,\ldots, n$ is
\[
(\Rz_i f)(x) = c \, \mathrm{P.V.} \int_{\R^n} \frac{x_i - y_i}{|x-y|^{n+1}} f(y)\,dy.
\]
The notation $\mathrm{P.V.}$ stands for the principal value of the integral. 
For the Riesz potential, we have
\[
(I^s f)(x) = c \, \mathrm{P.V.} \int_{\R^n} \frac{f(y)}{|x-y|^{n-s}}\,dy
\]
for $0 < s < n$. 
For $s \in (0,2)$, the fractional Laplacian $\laps{s}$ is defined by
\[
\laps{s} f(x) = c \, \mathrm{P.V.} \int_{\R^n} \frac{f(x)-f(y)}{|x-y|^{n+s}}\, dy.
\]

Next we introduce the fractional Sobolev space $W^{s,p}(\R^n)$ for $s \in (0,1)$. This space is induced by the semi-norm (called Sobolev-Slobodeckij or Gagliardo norm)
\[
[f]_{W^{s,p}(\Omega)} = \brac{\int_{\Omega} \int_{\Omega} \frac{|f(x)-f(y)|^{p}}{|x-y|^{n+sp}}\, dx\, dy}^{\frac{1}{p}},
\]
and $\|\cdot\|_{W^{s,p}(\Omega)} = \| \cdot\|_{L^{p}(\Omega)} +  [\cdot]_{W^{s,p}(\Omega)}$ serves as a norm.

%For $s_1,s_2 \in (0,2)$ so that $s_1 + s_2 = 2s$ using the inverse relationship between the fractional Laplacian and the Riesz potential
%\[
%u(x) = \lapms{s_1} \laps{s_1} u(x) = \int_{\R^n} |x-z|^{s_1-n}\, \laps{s_1} u(z)\, dz
%\]
%and
%\[
%u(y) = \lapms{s_1} \laps{s_1} u(y) = \int_{\R^n} |y-z|^{s_1-n}\, \laps{s_1} u(z)\, dz.
%\]
%Thus,
%\[
%u(x)-u(y) = \int_{\R^n} \brac{|x-z_1|^{s_1-n}-|y-z_1|^{s_1-n}}\, \laps{s_1} u(z_1)\, dz_1
%\]
%Similarly, 
%\[
%\varphi(x)-\varphi(y) = \int_{\R^n} \brac{|x-z_2|^{s_2-n}-|y-z_2|^{s_2-n}}\, \laps{s_2} \varphi(z_2)\, dz_2
%\]
%Then
%\[
%\begin{split}
%&\int_{\R^n}\int_{\R^n} K(x,y) \frac{(u(x)-u(y))(\varphi(x)-\varphi(y))}{|x-y|^{n+2s}}\, dx\, dy\\
%= &\int_{\R^n} \int_{\R^n} A(z_1,z_2) \laps{s_1} u(z_1)\, \laps{s_2} \varphi(z_2)\, \ dz_1\, dz_2
%\end{split}
%\]
%where
%\[
%	A(z_1,z_2) = \int_{\R^n}\int_{\R^n} K(x,y) \frac{(|x-z_1|^{s_1-n}-|y-z_1|^{s_1-n})(|x-z_2|^{s_2-n}-|y-z_2|^{s_2-n})}{|x-y|^{n+2s}}\, dx\, dy.
%\]
%
%Using integration by parts, we obtain
%\[
%\int_{\R^n} f(z)\, \varphi(z)\,dz = \int_{\R^n} f(z)\, \lapms{s_2}\laps{s_2}\varphi(z)\,dz = \int_{\R^n} \lapms{s_2}f(z)\, \laps{s_2}\varphi(z)\,dz.
%\]
%
%Thus, we can rewrite the equation \eqref{eq:nonloc} as
%\[
%\int_{\R^n} \int_{\R^n} A(z_1,z_2)\ \laps{s_1} u(z_1)\, \laps{s_2} \varphi(z_2)\,  dz_1\, dz_2 = \int_{\R^n} \lapms{s_2}f(z)\, \laps{s_2}\varphi(z)\,dz.
%\]
%
%Since %$\laps{s_2}$ is elliptic and 
%this equation holds for all $\varphi$, %(in distributional sense) 
%we get
%\[
%\int_{\R^n} A(z_1,z_2)\, \laps{s_1} u(z_1)\, dz_1 = \lapms{s_2} g(z_2)
%\]
%as desired.

We now show that $T_{K,s_1,s_2}$ is bounded from $L^2$ to $L^2$. 
\begin{proposition}\label{la:L2bound}
Let $K \in L^\infty(\R^n\times \R^n)$. 
Then, for any $f \in L^2(\R^n)$, we have
\[
 \|T_{K,s_1,s_2} f\|_{L^2(\R^n)} \aleq \|K\|_{L^\infty(\R^n\times \R^n)} \|f\|_{L^2(\R^n)}.
\]
\end{proposition}
\begin{proof} 
Let $\varphi \in C^\infty_c(\R^n)$. By Fubini's theorem, 
\[
\int_{\R^n} (T_{K,s_1,s_2} f)(z_1) \varphi(z_1) \,dz_1 
%= \int_{\R^n} \int_{\R^n} \int_{\R^n} \int_{\R^n} K(x,y) \frac{\brac{|x-z_1|^{s_1-n}-|y-z_1|^{s_1-n}} \brac{|x-z_2|^{s_2-n}-|y-z_2|^{s_2-n}}}{|x-y|^{n+2s}} f(z_2) \varphi(z_1)  \, dx\, dy\,dz_2\,dz_1\\
= \int_{\R^n} \int_{\R^n} K(x,y)\frac{(\lapms{s_1} \varphi (x)-\lapms{s_1} \varphi(y))\, (\lapms{s_2} f(x)-\lapms{s_2} f(y))}{|x-y|^{n+2s}}\, dx\, dy.
\]

Now, using H\"older's inequality twice, we obtain that (recall that $s_1+s_2 = 2s$ and $s_1,s_2 \in (0,1)$,
\[
\begin{split}
&\int_{\R^n} (T_{K,s_1,s_2} f)(z_1) \varphi(z_1) \,dz_1 \\
&= \int_{\R^n} \int_{\R^n} K(x,y)\frac{(\lapms{s_1} \varphi (x)-\lapms{s_1} \varphi(y))\, (\lapms{s_2} f(x)-\lapms{s_2} f(y))}{|x-y|^{n+2s}}\, dx\, dy\\
&\aleq \int_{\R^n} \left( \int_{\R^n} \frac{|\lapms{s_1} f(x)-\lapms{s_1} f(y)|^2}{|x-y|^{n+2s_1}}\, dx \right)^\frac{1}{2} \, \left( \int_{\R^n} \frac{|\lapms{s_1} \varphi(x)-\lapms{s_1} \varphi(y)|^2}{|x-y|^{n+2s_2}}\, dx \right)^\frac{1}{2} \,dy \\
&\aleq \left( \int_{\R^n} \int_{\R^n} \frac{|\lapms{s_1} f(x)-\lapms{s_1} f(y)|^2}{|x-y|^{n+2s_1}}\,dx\,dy \right)^\frac{1}{2} 
 \, \left( \int_{\R^n} \int_{\R^n} \frac{|\lapms{s_2} \varphi(x)-\lapms{s_2} \varphi(y)|^2}{|x-y|^{n+2s_2}}\, dx\,dy \right)^\frac{1}{2} \\
&\aeq [\lapms{s_1} f]_{W^{s_1,2}(\R^n)}\, [\lapms{s_2} \varphi]_{W^{s_1,2}(\R^n)}. 
\end{split}
\]
We now employ Triebel-Lizorkin space theory: $\dot{W}^{s,p} \aeq \dot{F}^{s}_{pp}$, and $\lapms{s} \dot{F}^{s}_{pp} \aeq \dot{F}^0_{pp}$, see \cite[p.14]{RS96}, moreover $\dot{F}^0_{p2} \aeq L^p$  see \cite[Proposition 2, p. 95]{RS96}. In particular,
\[
 [\lapms{t} f]_{W^{t,2}} \aeq \|f\|_{L^2(\R^n)}.
\]
Hence, we get 
\[
\int_{\R^n} (T_{K,s_1,s_2} f)(z_1) \varphi(z_1) \,dz_1 \aleq \|f\|_{L^2(\R^n)} \|\varphi\|_{L^2(\R^n)}
\]
for any $\varphi \in C^\infty_c(\R^n)$. By duality implies that
\[
\|Tf\|_{L^2(\R^n)} = \sup_{\substack{\varphi \in C^\infty_c(\R^n)\\\|\varphi\|_{L^2(\R^n)} \leq 1}} \int (T_{K,s_1,s_2} f)(z_1) \varphi(z_1) \,dz_1
 \aleq  \|f\|_{L^2(\R^n)}. 
\]
\end{proof}

\section{\texorpdfstring{$A_{K,s_1,s_2}$}{A} is a standard kernel: Proof of Proposition~\ref{pr:se:1}}\label{s:SK}
In this section, we prove that  $A_{K,s_1,s_2}$ is a \emph{standard kernel} which is satisfying the size and regularity conditions, \Cref{pr:se:1}, cf. \cite[Definition 4.1.2]{GMF}.
%We have the following result. 
%\begin{theorem}\label{pr:se:1}
%	For any  $z_1, z_2 \in \R^n$, we have
%	\begin{equation}\label{eq:CZ:g1}
%	A_{K,s_1,s_2}(z_1,z_2) \aleq  \frac{\|K\|_{L^\infty}}{|z_1-z_2|^{n}}  
%	\end{equation}
%	and there exists $\alpha > 0$ such that 
%	\begin{equation}\label{eq:CZ:g2}
%	\left |A_{K,s_1,s_2}(z_1+h,z_2) - A_K(z_1,z_2) \right | \aleq  \frac{|h|^{\alpha}\|K\|_{L^\infty}}{|z_1-z_2|^{n+\alpha}}
%	\end{equation}
%	whenever $|h| < \frac{1}{2} |z_1-z_2|$ and 
%	\begin{equation}\label{eq:CZ:g3}
%	\left |A_{K,s_1,s_2}(z_1,z_2+h) - A_K(z_1,z_2) \right| \aleq  \frac{|h|^{\alpha}\|K\|_{L^\infty}}{|z_1-z_2|^{n+\alpha}}
%	\end{equation}
%	whenever $|h| < \frac{1}{2} |z_1-z_2|$.
%\end{theorem}
%
%Observe that $A_{K,s_1,s_2}$ may not be symmetric in general (unless $s_1=s_2=s$). However, since for our setup the values of $s_1$ and $s_2$ are interchangeable, \eqref{eq:CZ:g3} and \eqref{eq:CZ:g2} are equivalent.
Before proving that, we give some estimates which will be useful later. 
The following lemma is a well-known application of the fundamental theorem of calculus, see, e.g., \cite[Lemma 3.3]{MSY20}.
\begin{lemma}\label{la:fundthm}
	For any $r \in\R$,  $\sigma \in [0,1]$ there exists a constant $C$ depending on $r$ such that the following holds. 
	
	Let $a,b \in \R^n\backslash \{0\}$ with $|a-b|\aleq \min\{|a|,|b|\}$. Then
	\[
	\abs{ |a|^r - |b|^{r} } \leq C\, |a-b|^{\sigma}\, \min\left \{|a|^{r-\sigma},|b|^{r-\sigma}\right \}.
	\]
\end{lemma}
For the convenience of the reader we repeat the proof.
\begin{proof}
	If $r=0$, the inequality follows trivially. Now suppose $r \neq 0$. 
	Observe that 
	if	$|a-b|\aleq \min\{|a|,|b|\}$, 
	%$|a-b| \leq \frac{1}{2}|b|$ or $|a-b| \leq \frac{1}{2} |a|$ 
	then $|a| \aeq |b|$ (with a uniform constant) and so, 
	\[
	\min\left \{|a|^{r-\sigma},|b|^{r-\sigma}\right \} \aeq |a|^{r-\sigma}.
	\]
	Note that for any $\sigma \in [0,1]$ we have
	\[
	|a-b| = |a-b|^{\sigma} |a-b|^{1-\sigma} \aleq |a-b|^{\sigma} |a|^{1-\sigma}.
	\]
	Thus, 
	%Using the above inequality, to complete the proof 
	it suffices to show that 
	\[
	\abs{ |a|^r - |b|^{r} } \aleq |a-b|\, |b|^{r-1}.
	\]
	Divide by $|b|^{r}$, the above inequality is equivalent to showing 
	\[
	\abs{ \abs{\frac{a}{|b|}}^r - \abs{\frac{b}{|b|}}^{r} } \aleq  \abs{\frac{a}{|b|}-\frac{b}{|b|}}.
	\]
	Since $|a| \aeq |b|$, there are uniform constants $0<r_1 <1< r_2<\infty$ such that both $\frac{a}{|b|}$ and ${\frac{b}{|b|}}$ are in $A:= B_{R_2}(0) \backslash B_{R_1}(0)$. Hence, we actually need to show that
	%the problem is now reduced to showing 
	\[
	\abs{ \abs{u}^r - \abs{v}^{r} } \leq C\, \abs{u-v} 
	\]
	for all $u,v \in A$. 
	% But this last inequality follows directly from the Mean Value Theorem, after noting that the function $f(u) = |u|^{r}$ is a smooth function on $A$ for any $r\in \R$.
	Since $A$ is an annulus, for any $u,v \in A$ there exists a curve $\gamma \subset A$ with $\gamma(0) = u$, $\gamma(1) = v$, $|\gamma'| \aeq  |u-v|$ -- with constants depending only on $r_1$ and $r_2$. %(and thus uniform). 
	We define $\eta : [0,1] \rightarrow \R$ by 
	\[
	\eta(t) := |\gamma(t)|^r.
	\]
	Then, the fundamental theorem of calculus implies
	\[
	\abs{ \abs{u}^r - \abs{v}^r } \leq \sup_{t \in [0,1]} |\eta'(t)| \aleq |\gamma(t)|^{r-1} |\gamma'(t)| \aleq |u-v|.
	\]
\end{proof}

Similarly in spirit to \Cref{la:fundthm}, the next lemma is also obtained by mean value theorem, albeit with quite a few more technical arguments.
\begin{lemma}\label{la:mvf}
For any $\alpha, \sigma \in [0,1]$ there exists a constant $C >0$ such that the following holds.

Let $a,b,h \in \R^n \backslash \{0\}$ such that $a+h, b+h \neq 0$. 
\begin{itemize}
	\item[(1)] If $|h| < \frac{1}{2} \min \{|a|,|b|\}$ or $|h| < \frac{1}{2} \min \{|a+h|,|b+h|\}$, 
	we have 
	\[
	\begin{split}
	&\left | |a+h|^{s-n}- |b+h|^{s-n} - \brac{|a|^{s-n}-|b|^{s-n}} \right | \\ 
	&\leq C |h|^\alpha \brac{\left | |a+h|^{s-\alpha-n}- |b+h|^{s-\alpha-n}\right | + \abs{|a|^{s-\alpha-n}-|b|^{s-\alpha-n}}}\\
	& +C|h|^\alpha \min\{|a|^{s-\alpha-\sigma-n},|b|^{s-\alpha-\sigma-n}\}\, |a-b|^\sigma
	\end{split}
	\]
	\item[(2)] If $|h| > \frac{1}{2} \min \{|a|,|b|\}$ and $|h| > \frac{1}{2} \min \{|a+h|,|b+h|\}$, then
	\[
	\begin{split}
	&\left | |a+h|^{s-n}- |b+h|^{s-n} \right | + \left |\brac{|a|^{s-n}-|b|^{s-n}} \right |\\ 
	&\leq C
	|h|^\alpha \abs{|a+h|^{s-\alpha-n}-|b+h|^{s-\alpha-n}}
	+ C|h|^\alpha \abs{|a|^{s-\alpha-n}-|b|^{s-\alpha-n}}.
	\end{split}
	\]
	\end{itemize}
\end{lemma}
\begin{proof}
\underline{(1) Assume $|h| < \frac{1}{2} \min \{|a|,|b|\}$ or $|h| < \frac{1}{2} \min \{|a+h|,|b+h|\}$}. Observe that in this case $|a| \aeq |a+\tilde{h}|$ and $|b| \aeq |b+\tilde{h}|$ for any $|\tilde{h}| \leq |h|$. In the following we thus may assume w.l.o.g. $|h| < \frac{1}{2} \min \{|a|,|b|\}$, the other case follows exactly the same way.

	Let $f: \R^n \rightarrow \R$ be defined by
	\[
	f(h) := |a+h|^{s-n}- |b+h|^{s-n}.
	\]
	By the mean value theorem, we have that for some $|\tilde{h}| < |h|$,
	\[
	\left |f(h) - f(0) \right| \aleq |h| |Df(\tilde{h})|.
	\]
	Now,
	\[
	\begin{split}
	Df(\tilde{h}) =& |a+\tilde{h}|^{s-1-n} \frac{a+\tilde{h}}{|a+\tilde{h}|} -|b+\tilde{h}|^{s-1-n} \frac{b+\tilde{h}}{|b+\tilde{h}|}\\
	=& \brac{|a+\tilde{h}|^{s-1-n} - |b+\tilde{h}|^{s-1-n}} \frac{a+\tilde{h}}{|a+\tilde{h}|} +|b+\tilde{h}|^{s-1-n} \brac{\frac{a+\tilde{h}}{|a+\tilde{h}|}-\frac{b+\tilde{h}}{|b+\tilde{h}|}}.\\
	\end{split}
	\]
	%For the second term, we observe that for any $c, d \in \R^n$, we have 
	We first treat the second term in this estimate. Note that for any $c, d \in \R^n\backslash \{0\}$, we have
	\[
	\begin{split}
	\left |\frac{c}{|c|} - \frac{d}{|d|} \right | =& \frac{1}{|c|\, |d|} \big | c|d| - d|c| \big |\\
	\leq& \frac{1}{|c|\, |d|} \brac{| c-d||d| +|d| \left ||d| - |c| \right |}\\
	\leq&\frac{2}{|c|} |c-d|.
	\end{split}
	\]
	We can interchange the role of $c$ and $d$, and so 
	\[
	\left|\frac{c}{|c|} - \frac{d}{|d|}\right| \aleq \min\{|c|^{-1},|d|^{-1}\} |c-d|
	\]
	This implies that for any $\sigma \in [0,1]$,
	\[
	\left|\frac{c}{|c|} - \frac{d}{|d|}\right| = \left|\frac{c}{|c|} - \frac{d}{|d|}\right|^\sigma\, \left|\frac{c}{|c|} - \frac{d}{|d|}\right|^{1-\sigma} \aleq 2^{1-\sigma}\, \min\{|c|^{-\sigma},|d|^{-\sigma}\} |c-d|^\sigma.
	\]
	That is, we have 
	\[
	|h| |Df(\tilde{h)}| \aleq |h|\left ||a+\tilde{h}|^{s-1-n} - |b+\tilde{h}|^{s-1-n}\right |+ |h| |b|^{s-1-\sigma-n}\, |a-b|^\sigma
	\]
	Since $|h| \aleq |b|$ this implies for any $\alpha \in [0,1]$,
	\[
	|h| |Df(\tilde{h)}| \aleq |h|\left ||a+\tilde{h}|^{s-1-n} - |b+\tilde{h}|^{s-1-n}\right |+ |h|^\alpha |b|^{s-\alpha-\sigma-n}\, |a-b|^\sigma.
	\]
	Interchanging the role of $|a|$ and $|b|$ we get 
	\[
	|h| |Df(\tilde{h)}| \aleq |h|\left ||a+\tilde{h}|^{s-1-n} - |b+\tilde{h}|^{s-1-n}\right |+ |h|^\alpha \min\{|a|^{s-\alpha-\sigma-n},|b|^{s-\alpha-\sigma-n}\}\, |a-b|^\sigma.
	\]
	It remains to estimate the first term. For this, we define $g: (0,\infty) \rightarrow \R$ by
	\[
	g(t) := t^{\frac{s-n-1}{s-n-\alpha}} = t^{\frac{n+1-s}{n+\alpha-s}}.
	\]
	Since $\alpha \in [0,1]$, observe that $\frac{n+1-s}{n+\alpha-s} \geq 1$.
	Let $t_1, t_2 \in (0,\infty)$. Then, by mean value theorem, there exists $c \in (t_1, t_2)$ such that
	\[
	|g(t_1)-g(t_2)| \leq g'(c) |t_1-t_2|.
	\]
	This implies that 
	\[
	|g(t_1)-g(t_2)| \aleq \max\{(t_1)^{\frac{n+1-s}{n+\alpha-s} -1},(t_2)^{\frac{n+1-s}{n+\alpha-s} -1}\}\, |t_1-t_2|.
	\]
	We apply the above inequality with $t_1=|a+\tilde{h}|^{s-n-\alpha} \aeq |a|^{s-n-\alpha}$ and $t_2=|b+\tilde{h}|^{s-n-\alpha} \aeq |b|^{s-n-\alpha}$. Then we get
	\[
	\begin{split}
	&\left ||a+\tilde{h}|^{s-1-n} - |b+\tilde{h}|^{s-1-n}\right |\\
	&=\left |g\brac{|a+\tilde{h}|^{s-n-\alpha}} - g\brac{|b+\tilde{h}|^{s-n-\alpha}}\right |\\
	&\aleq \max\left \{ |a|^{\brac{s-n-\alpha}\brac{\frac{n+1-s}{n+\alpha-s} -1}},|b|^{\brac{s-n-\alpha}\brac{\frac{n+1-s}{n+\alpha-s} -1}}\right \}\, \left ||a+\tilde{h}|^{s-\alpha-n} - |b+\tilde{h}|^{s-\alpha-n}\right |\\
	&\aeq\max\left \{ |a|^{\alpha-1},|b|^{\alpha-1}\right \}\, \left ||a+\tilde{h}|^{s-\alpha-n} - |b+\tilde{h}|^{s-\alpha-n}\right |.\\
	\end{split}
	\]
	Since $\alpha \leq 1$, $\max\{ |a|^{\alpha-1},|b|^{\alpha-1}\} = \min\{|a|,|b|\}^{\alpha-1}$. That is, again since $|h| < \frac{1}{2} \min \{|a|,|b|\}$ we have 
	\[
	|h|\, \left ||a+\tilde{h}|^{s-1-n} - |b+\tilde{h}|^{s-1-n}\right | \aleq |h|^\alpha \left ||a+\tilde{h}|^{s-\alpha-n} - |b+\tilde{h}|^{s-\alpha-n}\right |.
	\]
	This concludes the first claim.
% 	
% 	If $|h| < \frac{1}{2} \min \{|a+h|,|b+h|\}$, then we also get the same result because $|a| \aeq |a+\tilde{h}|$ and $|b| \aeq |b+\tilde{h}|$ for any $|\tilde{h}| \leq |h|$.
% 	
	
	\underline{(2) Assume that $|h| > \frac{1}{2} \min \{|a|,|b|\}$ \emph{and} $|h| > \frac{1}{2} \min \{|a+h|,|b+h|\}$}. 
	
	We only show the estimate of $\left | |a+h|^{s-n}- |b+h|^{s-n} \right |$, the estimate for $\left | |a|^{s-n}- |b|^{s-n} \right |$ is almost verbatim. 
	
	We have two cases.
	 
	\underline{Case 1: $\min \{|a+h|,|b+h|\} \leq \frac{1}{2} \max \{|a+h|,|b+h|\}$}.
	Then for any $\theta \in (0,n)$, with a constant only depending on $\theta-n$,
	\[
	\left | |a+h|^{\theta-n}- |b+h|^{\theta-n} \right | \aeq \min \{|a+h|,|b+h|\}^{\theta-n}
	\]
	Thus,
	\[
	\begin{split}
	\left | |a+h|^{s-n}- |b+h|^{s-n} \right | 
	&\aeq \min\{|a+h|, |b+h|\}^{s-n}\\
	&= \min\{|a+h|, |b+h|\}^{s-n-\alpha+\alpha}\\
	&\aleq |h|^\alpha \, \min\{|a+h|, |b+h|\}^{s-n-\alpha}\\
	&\aeq |h|^\alpha \, \left ||a+h|^{s-\alpha-n} - |b+h|^{s-\alpha-n}\right |.
	\end{split}
	\]
	\underline{Case 2: $\min \{|a+h|,|b+h|\} \aeq \max \{|a+h|,|b+h|\}$.}
	We define $g : [0,\infty) \rightarrow [0,\infty)$ by
	\[
	g(t) := t^{\frac{s-n}{s-n-\alpha}}
	\]
	for any $\alpha \in [0,1]$. Then, for any $t_1,t_2 \in (0,\infty)$,
	by mean value theorem, there exists $d \in (t_1,t_2)$ such that
	\[
	|g(t_1)-g(t_2)| \leq g'(d)\, |t_1-t_2|.
	\]
	Thus,
	%Since $\frac{s-n}{s-n-\alpha} \leq 1$,
	\[
	|g(t_1)-g(t_2)| \leq \max\{t_1^{\frac{s-n}{s-n-\alpha} -1},t_2^{\frac{s-n}{s-n-\alpha} -1}\}\, |t_1-t_2|.
	\]
	For $t_1 = |a+h|^{s-n-\alpha}$ and $t_2 = |b+h|^{s-n-\alpha}$ we then find
	\[
	\begin{split}
	&\left | |a+h|^{s-n}- |b+h|^{s-n} \right | \\
	&= \left |g\brac{|a+h|^{s-n-\alpha}} - g\brac{|b+h|^{s-n-\alpha}}\right |\\
	&\aleq \max\{|a+h|^{(s-n-\alpha)\brac{\frac{s-n}{s-n-\alpha} -1}}, |b+h|^{(s-n-\alpha)\brac{\frac{s-n}{s-n-\alpha} -1}} \}\, \left ||a+h|^{s-\alpha-n} - |b+h|^{s-\alpha-n}\right |\\
	&\aeq \max\{ |a+h|^\alpha,|b+h|^\alpha\}\, \left ||a+h|^{s-\alpha-n} - |b+h|^{s-\alpha-n}\right |\\
	&\aeq \min\{ |a+h|^\alpha,|b+h|^\alpha\}\, \left ||a+h|^{s-\alpha-n} - |b+h|^{s-\alpha-n}\right |\\
	&\aleq |h|^\alpha \, \left ||a+h|^{s-\alpha-n} - |b+h|^{s-\alpha-n}\right |.
	\end{split}
	\]
	This concludes the proof of the second claim.
\end{proof}

With the help of \Cref{la:mvf} we will be able to reduce the proof of \Cref{pr:se:1} to the estimate of \Cref{pr:se:1v2} below.
Set 
for $l = 1, 2$, $\alpha, \sigma \in [0,1]$, $s,s_1,s_2 \in (0,1)$ with $s_1+s_2 = 2s$,
\[
M^{\alpha, \sigma}_l(z_1, z_2) := \int_{\R^n} \int_{\R^n} K(x,y) \, \kappa^{\alpha, \sigma}_l(x,y,z_1,z_2) \, dx\, dy
\]
where
\[
\kappa^{\alpha, \sigma}_1(x,y,z_1,z_2) := \frac{\left | |x-z_1|^{s_1-\alpha-n}-|y-z_1|^{s_1-\alpha-n} \right |\, \left | |x-z_2|^{s_2-n}-|y-z_2|^{s_2-n}\right|}{|x-y|^{n+2s}}
\]
and
\[
\kappa^{\alpha,\sigma}_2(x,y,z_1,z_2) := \frac{
	\min\{|x-z_1|^{s_1-\alpha-\sigma-n},|y-z_1|^{s_1-\alpha-\sigma-n}\} 
	\left | |x-z_2|^{s_2-n}-|y-z_2|^{s_2-n}\right|}{|x-y|^{n+2s-\sigma}}.
\]

\begin{proposition}\label{pr:se:1v2}
	Let $\theta \in (0,\frac{1}{10})$ be such that $10 \theta < s,s_1,s_2 < 1-10\theta$. 
	Then, for all $\alpha \in [0,\frac{1}{10}\theta)$, $\sigma \in (s_1+\theta,2s)$  and for $l=1,2$ we have 
	\[
	|M^{\alpha, \sigma}_l(z_1, z_2)| \leq  C(\theta)\, \frac{\|K\|_{L^\infty(\R^n \times \R^n)}}{|z_1-z_2|^{\alpha+n}} \quad \mbox{ for all } z_1\neq z_2.
	\]
\end{proposition}
\begin{proof}
	\newcommand{\CA}{\mathcal{A}}
	\newcommand{\CB}{\mathcal{B}}
	\newcommand{\CI}{\mathcal{I}}
	\newcommand{\CJ}{\mathcal{J}}
	Fix $z_1, z_2 \in \R^n$ and let $\delta := |z_1-z_2|>0$. 
	
	We split $\R^n \times \R^n$ into different cases
	\begin{equation}\label{eq:split:12}
	 \R^n \times \R^n = \bigcup_{i=1}^3 \CA_i = \bigcup_{i=1}^3 \CB_i = \bigcup_{i=1}^3 \CI_i,
	\end{equation}
where
	\[
	\begin{split}
	\CA_1 \equiv \CA_1(z_1) &= \{ (x, y) \in \R^n \times \R^n : |x-y| \leq 10\min\{|x-z_1|, |y-z_1|\} \},\\
	\CA_2\equiv\CA_2(z_1) &= \{ (x, y) \in \R^n \times \R^n : |x-z_1| \leq 10\min\{|y-z_1|, |x-y|\} \},\\
	\CA_3\equiv \CA_3(z_1) &= \{ (x, y) \in \R^n \times \R^n : |y-z_1| \leq 10\min\{|x-z_1|, |x-y|\} \},
	\end{split}
	\]
	and $\CB_i$ are the analogous cases involving $z_2$, namely
	\[
	\begin{split}
	\CB_1\equiv\CB_1(z_2) &= \{ (x, y) \in \R^n \times \R^n : |x-y| \leq 10\min\{|x-z_2|, |y-z_2|\} \},\\
	\CB_2\equiv\CB_2(z_2) &= \{ (x, y) \in \R^n \times \R^n : |x-z_2| \leq 10\min\{|y-z_2|, |x-y|\} \},\\
	\CB_3\equiv\CB_3(z_2) &= \{ (x, y) \in \R^n \times \R^n : |y-z_2| \leq 10\min\{|x-z_2|, |x-y|\} \},
	\end{split}
	\]
	and lastly $\CI_i$,	
	\[
	\begin{split}
	\CI_1 \equiv\CI_1(z_1,z_2) &= \{ (x, y) \in \R^n \times \R^n : y \in \R^n,|x-z_1| \leq 10 \delta \mbox{ and } |x-z_2| \geq \frac{1}{10}\delta \},\\
	\CI_2\equiv\CI_2(z_1,z_2) &= \{ (x, y) \in \R^n \times \R^n : y \in \R^n, |x-z_2| \leq 10\delta \mbox{ and } |x-z_1| \geq \frac{1}{10}\delta \},\\
	\CI_3\equiv\CI_3(z_1,z_2) &= \left \{ (x, y) \in \R^n \times \R^n : y \in \R^n, \frac{1}{100} |x-z_2| \leq |x-z_1| \leq 100|x-z_2| \mbox{ and } |x-z_1| \geq \frac{1}{100}\delta \right \}.
	\end{split}
	\]
It is an elementary exercise to establish \eqref{eq:split:12} (recall that $|z_1-z_2| = \delta$). Observe that there is no need for the sets to be disjoint.	
	
	We write $M^{\alpha, \sigma}_l(z_1, z_2)$ as
	
	\[
	\begin{split}
	M^{\alpha, \sigma}_l(z_1, z_2) 
	&\le \sum_{i,j,k=1}^3 \iint_{\CA_i \cap \CB_j \cap \CI_k} K(x,y) \, \kappa^{\alpha, \sigma}_l(x,y,z_1,z_2) \, dx\, dy\\
	&=: \sum_{i,j,k=1}^3 J^{\alpha,\sigma,l}_{i,j,k}(z_1,z_2).
	\end{split}
	\]
% 	
% 	The above splitting depends on how far $x$ with $z_1$ as figure:
% 	
% 	
% 
% 	
% 	
%	\begin{remark}
%		Observe that if 
%		$$|x-y| \aleq \min\{|x-z_1|, |y-z_1|\}$$ 
%		then
%		\[
%		\max\{|x-z_1|, |y-z_1|\} \le 2 \min\{|x-z_1|, |y-z_1|\}
%		\]
%		
%		which implies that $|x-z_1|$ and $|y-z_1|$ are comparable.
%		%for any $x, y \in \R^n$, then $|x-z_1|$ and $|y-z_1|$ are comparable.
%	\end{remark}
% 	
% 	
	Our strategy is now to consider all combination of the cases above seperately. That is, we prove below that $$J^{\alpha,\sigma,l}_{i,j,k}(z_1, z_2) \aleq \delta^{-\alpha-n}$$
	for all $i,j,k =1,2,3$ and all $l=1,2$.
	
	\underline{Estimating $J^{\alpha, \sigma, l}_{1,1,1}$, $J^{\alpha, \sigma, l}_{1,1,2}$ and $J^{\alpha, \sigma, l}_{1,1,3}$} :
	we begin noting that for $(x, y) \in \CA_1$ and $(x, y) \in \CB_1$ we have $|x-z_1| \aeq |y-z_1|$ and $|x-z_2| \aeq |y-z_2|$. Moreover,  \Cref{la:fundthm} leads to
	\[
	\left| |x-z_1|^{s_1-\alpha-n} - |y-z_1|^{s_1-\alpha-n} \right| \aleq |x-z_1|^{s_1-\alpha-n-1}\,|x-y|
	\]
	and
	\[
	\left| |x-z_2|^{s_2-n} - |y-z_2|^{s_2-n} \right| \aleq |x-z_2|^{s_2-n-1}\,|x-y|.
	\]
	
	Thus, for $(x, y) \in \CA_1 \cap \CB_1$, we have
	\[ 
	\kappa^{\alpha, \sigma}_1(x,y,z_1,z_2)  \aleq 
	 \frac{|x-z_1|^{s_1-\alpha-n-1}\, |x-z_2|^{s_2-n-1}}{|x-y|^{n+2s-2}}.
	\]
	We can integrate in $y$, observe that $(x,y) \in \CA_1 \cap \CB_1$ implies that $|x-y| \aleq \min\{|x-z_1|,|x-z_2|\}$, and thus (since $s < 1$)
	\[
	 \int_{y \in \CA_1 \cap \CB_1} \frac{1}{|x-y|^{n+2s-2}} dy\aleq \min\{|x-z_1|,|x-z_2|\}^{2-2s}.
	\]
That is, after integrating in $y$, we get 
	
	\[
	\begin{split}
	J^{\alpha, \sigma, 1}_{1,1,1}(z_1, z_2)
	&\aleq \int_{\CA_1 \cap \CB_1 \cap \CI_1} |x-z_1|^{s_1-\alpha-n-1}\, |x-z_2|^{s_2-1-n}\,|x-z_1|^{2-2s}\, dx\\
	&= \int_{\CA_1 \cap \CB_1 \cap \CI_1} |x-z_1|^{1-\alpha-s_2-n}\, |x-z_2|^{s_2-1-n}\, dx.
	\end{split}
	\]
	Now we use H\"older's inequality: let $p > 1$ be so small such that $(1-\alpha-s_2-n)p > -n$ and $(s_2-1-n)p'<-n$, then
	\[
	J^{\alpha, \sigma, 1}_{1,1,1}(z_1, z_2)
	\aleq \brac{\int_{|x-z_1| \aleq \delta} |x-z_1|^{(1-\alpha-s_2-n)p}\, dx}^{\frac{1}{p}}\, \brac{\int_{|x-z_2|\ageq \delta} |x-z_2|^{(s_2-1-n)p'}\, dx}^{\frac{1}{p'}}.
	\]
	Using that $(1-\alpha-s_2-n)p > -n$ and $(s_2-1-n)p'<-n$, we compute
	\[
	J^{\alpha, \sigma, 1}_{1,1,1}(z_1, z_2)
	\aleq \brac{\delta^{(1-\alpha-s_2-n)p+n}}^{\frac{1}{p}}\, \brac{\delta^{(s_2-1-n)p'+n}}^{\frac{1}{p'}} = \delta^{-\alpha-n}
	\]
	which settles this case. 
	
	We argue similarly for $J^{\alpha, \sigma, 1}_{1,1,2}(z_1, z_2)$ (essentially only interchanging the role of $z_1$ and $z_2$ in the argument above).
	
	To estimate $J^{\alpha, \sigma, 1}_{1,1,3}(z_1, z_2)$, we also argue similarly, but since $|x-z_1| \aeq |x-z_2|$ when $(x,y) \in \CI_3$, we arrive at
	\[
	\begin{split}
	J^{\alpha, \sigma, 1}_{1,1,3}(z_1, z_2)
	&\aleq \int_{\CA_1 \cap \CB_1 \cap \CI_3} |x-z_1|^{1-\alpha-s_2-n}\, |x-z_1|^{s_2-1-n}\, dx \\
	&= \int_{|x-z_1| \ageq \delta} |x-z_1|^{-\alpha-2n}  \,dx \\
	&\aeq \delta^{-\alpha-n}.
	\end{split}
	\]

	Next, let us estimate for $J^{\alpha, \sigma, 2}_{i,j,k}$. By \Cref{la:fundthm} applied to the $z_2$-term, we obtain that
	\[
	\kappa^{\alpha, \sigma}_2(x,y,z_1,z_2)  \aleq \frac{|x-z_1|^{s_1-\sigma-\alpha-n}\, |x-z_2|^{s_2-n-1}}{|x-y|^{n+2s-\sigma-1}}.
	\]
	By assumption, we have $\sigma+1 > 2s$ and so
	we can integrate w.r.t. $y$ to get
	\[
	\begin{split}
	J^{\alpha, \sigma, 2}_{1,1,1}(z_1, z_2)
	&\aleq \int_{\CA_1 \cap \CB_1 \cap \CI_1} |x-z_1|^{s_1-\sigma-\alpha-n}\, |x-z_2|^{s_2-n-1}\,|x-z_1|^{-2s+\sigma+1}\, dx\\
	&= \int_{\CA_1 \cap \CB_1 \cap \CI_1} |x-z_1|^{-s_2-\alpha+1-n}\, |x-z_2|^{s_2-1-n}\, dx.
	\end{split}
	\]
	Observe that $- s_2 -\alpha +1> 0$ and $s_2 - 1 < 0$ by assumption. 
	Let $p > 1$ be so small such that $(-s_2-\alpha+1-n)p > -n$ and $(s_2-1-n)p'<-n$. Then, by H\"older's inequality, we have
	\[
	\begin{split}
	J^{\alpha, \sigma, 2}_{1,1,1}(z_1, z_2)
	&\aleq \brac{\int_{|x-z_1| \aleq \delta} |x-z_1|^{(-s_2-\alpha+1-n)p}\, dx}^{\frac{1}{p}}\, \brac{\int_{|x-z_2|\ageq \delta} |x-z_2|^{(s_2-1-n)p'}\, dx}^{\frac{1}{p'}}\\
	&\aleq \brac{\delta^{(-s_2-\alpha+1-n)p+n}}^{\frac{1}{p}}\, \brac{\delta^{(s_2-1-n)p'+n}}^{\frac{1}{p'}}\\
	&= \delta^{-\alpha-n}.
	\end{split}
	\]
	Similarly, for $J^{\alpha, \sigma, 2}_{1,1,2}(z_1, z_2)$, 
	after integrating w.r.t. $y$ (since $\sigma+1 > 2s$) we get
	\[
	\begin{split}
	J^{\alpha, \sigma, 2}_{1,1,2}(z_1, z_2)
	&\aleq \int_{\CA_1 \cap \CB_1 \cap \CI_2} |x-z_1|^{s_1-\sigma-\alpha-n}\, |x-z_2|^{s_2-n-1}\,|x-z_2|^{-2s+\sigma+1}\, dx\\
	&= \int_{\CA_1 \cap \CB_1 \cap \CI_2} |x-z_1|^{s_1-\sigma-\alpha-n}\, |x-z_2|^{-s_1+\sigma-n}\, dx.
	\end{split}
	\]	
	By assumption $s_1 - \sigma - \alpha < 0$ and $-s_1+\sigma > 0$.
	Then, using H\"older's inequality, we let $p > 1$ be so small such that $(s_1-\sigma-\alpha-n)p < -n$ and $(-s_1+\sigma-n)p'> -n$. Then,
	\[
	\begin{split}
	J^{\alpha, \sigma, 2}_{1,1,2}(z_1, z_2)
	&\aleq \brac{\int_{|x-z_1| \ageq \delta} |x-z_1|^{(s_1-\sigma-\alpha-n)p}\, dx}^{\frac{1}{p}}\, \brac{\int_{|x-z_2|\aleq \delta} |x-z_2|^{(-s_1+\sigma-n)p'}\, dx}^{\frac{1}{p'}}\\
	&\aleq \brac{\delta^{(s_1-\sigma-\alpha-n)p+n}}^{\frac{1}{p}}\, \brac{\delta^{(-s_1+\sigma-n)p'+n}}^{\frac{1}{p'}}\\
	&= \delta^{-\alpha-n}.
	\end{split}
	\]

	To estimate $J^{\alpha, \sigma, 2}_{1,1,3}(z_1, z_2)$, we also argue similarly, but since $|x-z_1| \aeq |x-z_2|$ we arrive at
	\[
	\begin{split}
	J^{\alpha, \sigma, 2}_{1,1,3}(z_1, z_2)
	&\aleq \int_{\CA_1 \cap \CB_1 \cap \CI_3} |x-z_1|^{-s_2-\alpha+1-n}\, |x-z_2|^{s_2-n-1}\, dx \\
	&= \int_{|x-z_1| \ageq \delta} |x-z_1|^{-\alpha-2n}  \,dx \\
	&\aeq \delta^{-\alpha-n}.
	\end{split}
	\]

	\underline{Estimating $J^{\alpha, \sigma, l}_{1,2,1}$, $J^{\alpha, \sigma, l}_{1,2,2}$ and $J^{\alpha, \sigma, l}_{1,2,3}$} :
	
	if $(x, y) \in \CA_1$ we have $|x-z_1| \aeq |y-z_1|$ and thus we can estimate with \Cref{la:fundthm} for any $\gamma \in [0,1]$,
	\[||x-z_1|^{s_1-\alpha-n}-|y-z_1|^{s_1-\alpha-n}| \aleq |x-y|\, |x-z_1|^{s_1-\alpha-1-n} \aleq |x-y|^{\gamma} |x-z_1|^{s_1-\alpha-\gamma-n}.\] 
	For $(x,y) \in \CB_2$, we estimate $||x-z_2|^{s_2-n}-|y-z_2|^{s_2-n}| \aleq |x-z_2|^{s_2-n}$.
	
	Thus, for $(x, y) \in \CA_1 \cap \CB_2$ for any $\gamma \in [0,1]$
	\[
	\begin{split}
	\kappa^{\alpha, \sigma}_1(z_1,z_2)
	&\aleq \frac{|x-z_1|^{s_1-\alpha-\gamma-n}\, |x-z_2|^{s_2-n}}{|x-y|^{n+2s-\gamma}}.
	\end{split}
	\] 
	Since for $(x,y) \in \CA_1$ we have $|x-y| \aleq |x-z_1|$, we also get the same estimate for the second type kernel, for any $\gamma \in [0,\sigma)$
	\[
	\kappa^{\alpha, \sigma}_2(z_1,z_2)
	\aleq  \frac{|x-z_1|^{s_1-\alpha-\gamma-n}\, |x-z_2|^{s_2-n}}{|x-y|^{n+2s-\gamma}}.
	\]
	Taking $\gamma < 2s$, we integrate w.r.t. $y$ variable,  observing that $(x,y) \in \CA_1 \cap \CB_2$ implies that $|x-y| \ageq |x-z_2|$ and thus $\int_{y \in \CA_1\cap \CA_2} |x-y|^{-n-2s+\gamma} dy \aleq |x-z_2|^{-2s+\gamma}$.
	
	We thus obtain for any $\gamma \in [0,\sigma)$, $\gamma < 1$,
	\[
	J^{\alpha, \sigma, l}_{1,2,k}(z_1, z_2) \aleq \int_{\CI_k} |x-z_1|^{s_1-\alpha-\gamma-n}\, |x-z_2|^{s_2-n}\, |x-z_2|^{\gamma-2s}\, dx =
	\int_{\CI_k} |x-z_1|^{s_1-\alpha-\gamma-n}\, |x-z_2|^{\gamma-s_1-n}\, dx.
	\]
	Here, with a slight abuse of notation we identify $\CI_k$ with the set of $x \in \R^n$ such that $\{x\}\times \R^n \subset \CI_k$.
	
	Taking $\gamma = 0$ we thus have 
	\[
	\begin{split}
	J^{\alpha, \sigma, l}_{1,2,1}(z_1, z_2) 
 	\aleq&   	\int_{|x-z_1| \aleq \delta, |x-z_2| \ageq \delta} 
  	|x-z_1|^{s_1-\alpha-n}\, |x-z_2|^{-s_1-n}\, dx\\
  	\aleq& \delta^{-s_1-n}\int_{|x-z_1| \aleq \delta} 
  	|x-z_1|^{s_1-\alpha-n} \, dx\\
  	\aeq& \delta^{\alpha-n}.
  	\end{split}
	\]
    Similarly we can estimate (again, $\gamma = 0$)
    \[
	\begin{split}
	J^{\alpha, \sigma, l}_{1,2,3}(z_1, z_2) 
 	\aleq&   	\int_{|x-z_1| \aeq |x-z_2| \ageq \delta} 
  	|x-z_1|^{s_1-\alpha-n}\, |x-z_2|^{-s_1-n}\, dx\\
  	\aleq& \delta^{s_1-\alpha-n}\int_{|x-z_2| \ageq \delta} 
  	|x-z_2|^{-s_1-n} \, dx\\
  	\aeq& \delta^{\alpha-n}.
  	\end{split}
	\]
    For the remaining case $J^{\alpha, \sigma, l}_{1,2,2}$ we choose $\gamma > s_1$ (which is possible with the restraints on $\gamma$ above), and 
    have 
    \[
	\begin{split}
	J^{\alpha, \sigma, l}_{1,2,2}(z_1, z_2) 
 	\aleq&   	\int_{|x-z_1| \ageq \delta, |x-z_2| \aleq \delta} 
  	|x-z_1|^{s_1-\alpha-\gamma-n}\, |x-z_2|^{\gamma-s_1-n}\, dx\\
  	\aleq& \delta^{\gamma-s_1-n}\int_{|x-z_1| \ageq \delta} 
  	|x-z_1|^{s_1-\alpha-\gamma-n} \, dx\\
  	\aeq& \delta^{\alpha-n}.
  	\end{split}
	\]

	\underline{Estimating $J^{\alpha, \sigma, l}_{1,3,1}$, $J^{\alpha, \sigma, l}_{1,3,2}$ and $J^{\alpha, \sigma, l}_{1,3,3}$} :

	Let $(x, y) \in \CA_1 \cap \CB_3$. Then, by \Cref{la:fundthm} and $|x-y| \aleq |x-z_1|$,
	it follow that 
	\[
	\kappa^{\alpha, \sigma}_1(z_1,z_2)
	\aleq  \frac{|x-z_1|^{s_1-\alpha-1-n}\, |y-z_2|^{s_2-n}}{|x-y|^{n+2s-1}} \aleq 
	\frac{|x-z_1|^{s_1-\alpha-1-n}\, |y-z_2|^{s_2-n}}{|x-y|^{n+2s-1}}
	\] 
	and we have a similar estimate for the second kernel.
	\[
	\kappa^{\alpha, \sigma}_2(z_1,z_2)
	\aleq  \frac{|x-z_1|^{s_1-\alpha-\sigma-1-n}\, |y-z_2|^{s_2-n}}{|x-y|^{n+2s-\sigma-1}} \aleq 
	\frac{|x-z_1|^{s_1-\alpha-1-n}\, |y-z_2|^{s_2-n}}{|x-y|^{n+2s-1}}.
	\]
	We can treat these two kernels now almost verbatim. 
	Since $|x-y| \aleq |x-z_1|$ and $|x-z_2| \aeq |x-y|$ for $(x,y) \in \CA_1\cap\CB_3$, we have
	\[
	\begin{split}
	J^{\alpha, \sigma, l}_{1,3,1}(z_1,z_2) 
	&\aleq 
	\iint_{\CA_1 \cap \CB_3 \cap \CI_1} \frac{|x-z_1|^{s_1-\alpha-n}\, |y-z_2|^{s_2-n}}{|x-y|^{n+2s}}\, dx\, dy\\
	&\aeq \iint_{\CA_1 \cap \CB_3 \cap \CI_1} \frac{|x-z_1|^{s_1-\alpha-n}\, |y-z_2|^{s_2-n}}{|x-z_2|^{n+2s}}\, dx\, dy.
	\end{split}
	\]
	Since $|y-z_2| \aleq |x-z_2|$ and $s_2-n > -n$, we obtain $\int_{\CA_1 \cap \CB_3 \cap \CI_1} |y-z_2|^{s_2-n}\,dy \aleq |x-z_2|^{s_2}$ and so
	\[
	\begin{split}
	J^{\alpha, \sigma, l}_{1,3,1}(z_1,z_2) 
	&\aleq \int_{\CA_1 \cap \CB_3 \cap \CI_1} \frac{|x-z_1|^{s_1-\alpha-n}\, |x-z_2|^{s_2}}{|x-z_2|^{n+2s}} \, dx\\
	&\aeq \int_{\CA_1 \cap \CB_3 \cap \CI_1} |x-z_1|^{s_1-\alpha-n}\, |x-z_2|^{-s_1-n} \, dx.
	\end{split}
	\]
	
	Let $p>1$ so small such that $(s_1-\alpha-n)p > -n$ and $(-s_1-n)p' < -n$. Then, by H\"older's inequality,
	\[
	\begin{split}
	J^{\alpha, \sigma, l}_{1,3,1}(z_1,z_2) 
	&\aleq \left(\int_{\CA_1 \cap \CB_3 \cap \CI_1} |x-z_1|^{(s_1-\alpha-n)p}\ dx \right)^{\frac{1}{p}}\, \left(\int_{\CA_1 \cap \CB_3 \cap \CI_1} |x-z_2|^{(-s_1-n)p'}\ dx \right)^{\frac{1}{p'}}\\
	&\aleq \left(\delta^{(s_1-\alpha-n)p+n}\right)^\frac{1}{p}\,\left(\delta^{(-s_1-n)p'+n}\right)^\frac{1}{p'}\\
	&\aeq \delta^{-\alpha-n}.
	\end{split}
	\]

	For $J^{\alpha, \sigma, l}_{1,3,2}(z_1,z_2)$, we consider
	\[
	\begin{split}
	J^{\alpha, \sigma, l}_{1,3,2}(z_1,z_2) 
	&\aleq 
	\iint_{\CA_1 \cap \CB_3 \cap \CI_2} \frac{|x-z_1|^{s_1-\alpha-1-n}\, |y-z_2|^{s_2-n}}{|x-y|^{n+2s-1}}\, dx\, dy\\
	&\aeq \iint_{\CA_1 \cap \CB_3 \cap \CI_2} \frac{|x-z_1|^{s_1-\alpha-1-n}\, |y-z_2|^{s_2-n}}{|x-z_2|^{n+2s-1}}\, dx\, dy.
	\end{split}
	\]
	Since $|y-z_2| \aleq |x-z_2|$ and $s_2-n > -n$, we obtain $\int_{\CA_1 \cap \CB_3 \cap \CI_2} |y-z_2|^{s_2-n}\,dy \aleq |x-z_2|^{s_2}$ and so
	\[
	\begin{split}
	J^{\alpha, \sigma, l}_{1,3,2}(z_1,z_2) 
	&\aleq \int_{\CA_1 \cap \CB_3 \cap \CI_2} \frac{|x-z_1|^{s_1-\alpha-1-n}\, |x-z_2|^{s_2}}{|x-z_2|^{n+2s-1}} \, dx\\
	&\aeq \int_{\CA_1 \cap \CB_3 \cap \CI_2} |x-z_1|^{s_1-\alpha-1-n}\, |x-z_2|^{-s_1+1-n} \, dx.
	\end{split}
	\]
	
	Let $p>1$ so small such that $(s_1-\alpha-1-n)p < -n$ and $(-s_1+1-n)p' > -n$. Then, by H\"older's inequality,
	\[
	\begin{split}
	J^{\alpha, \sigma, l}_{1,3,2}(z_1,z_2) 
	&\aleq \left(\int_{\CA_1 \cap \CB_3 \cap \CI_2} |x-z_1|^{(s_1-\alpha-1-n)p}\ dx \right)^{\frac{1}{p}}\, \left(\int_{\CA_1 \cap \CB_3 \cap \CI_2} |x-z_2|^{(-s_1+1-n)p'}\ dx \right)^{\frac{1}{p'}}\\
	&\aleq \left(\delta^{(s_1-\alpha-1-n)p+n}\right)^\frac{1}{p}\,\left(\delta^{(-s_1+1-n)p'+n}\right)^\frac{1}{p'}\\
	&\aeq \delta^{-\alpha-n}.
	\end{split}
	\]

	Similarly, we have
	\[
	\begin{split}
	J^{\alpha, \sigma, l}_{1,3,3}(z_1,z_2) 
	&\aleq \int_{|x-z_1| \aeq |x-z_2| \ageq \delta} |x-z_1|^{s_1-\alpha-1-n}\, |x-z_2|^{-s_1+1-n} \, dx
	\aleq \delta^{-\alpha-n}.
	\end{split}
	\]

%-------------------------------- A1

	\underline{Estimating $J^{\alpha, \sigma, l}_{2,1,1}$, $J^{\alpha, \sigma, l}_{2,1,2}$ and $J^{\alpha, \sigma, l}_{2,1,3}$} :
	for $(x, y) \in \CB_1$, by \Cref{la:fundthm} we have 
	\[
	\abs{|x-z_2|^{s_2-n}-|y-z_2|^{s_2-n}} \aleq |x-z_2|^{s_2-n-\gamma}|x-y|^\gamma
	\]
	for any $\gamma\in [0,1]$. 
	Thus, for $(x, y) \in \CA_2 \cap \CB_1$, we have
	\[
	\kappa^{\alpha, \sigma}_1(z_1,z_2) \aleq   \frac{|x-z_1|^{s_1-\alpha-n}\, |x-z_2|^{s_2-n-\gamma}}{|x-y|^{n+2s-\gamma}}.
	\] 

	Since in this case $|y-z_1| \aeq |x-y|$, we also have
	\[
	\kappa^{\alpha, \sigma}_2(z_1,z_2) \aleq
	\frac{|x-z_1|^{s_1-\alpha-n}|y-z_1|^{-\sigma}\, |x-z_2|^{s_2-n-\gamma}}{|x-y|^{n+2s-\sigma-\gamma}}
	\aeq \frac{|x-z_1|^{s_1-\alpha-n} |x-z_2|^{s_2-n-\gamma}}{|x-y|^{n+2s-\gamma}}.
	\]
	Since $|x-z_1| \aleq |x-y|$, we can further obtain for any $t> 0$,
	\[
	\kappa^{\alpha, \sigma}_l(z_1,z_2) \aleq
	\frac{|x-z_1|^{s_1-\alpha-t-n} |x-z_2|^{s_2-n-\gamma}}{|x-y|^{n+2s-\gamma-t}}.
	\]
	To obtain the estimate when 
	$(x, y) \in \CA_2 \cap \CB_1 \cap \CI_1$, we let $t:= s_1 -5\theta$ and choose $\gamma=1$. Then $2s-1-t < 0$ and so we can integrate in $y$ to get
	\[
	\begin{split}
	J^{\alpha, \sigma, l}_{2,1,1}(z_1,z_2)
	&\aleq \iint_{\CA_2 \cap \CB_1 \cap \CI_1} \frac{|x-z_1|^{s_1-\alpha-t-n}\, |x-z_2|^{s_2-n-1}}{|x-y|^{n+2s-1-t}}\, dx\, dy\\
	&\aleq \int_{\CA_2 \cap \CB_1 \cap \CI_1} |x-z_1|^{s_1-\alpha-t-n}\, |x-z_2|^{s_2-n-1}\, |x-z_2|^{1+t-2s}\, dx\\
	&\aleq \int_{\CA_2 \cap \CB_1 \cap \CI_1} |x-z_1|^{s_1-\alpha-t-n}\, |x-z_2|^{-s_1+t-n}\, dx\\
	&\aleq \delta^{-\alpha-n}
	\end{split}
	\]
	where the last line follows from H\"older's inequality using that $s_1-\alpha-t > 0$ and $t -s_1 < 0$.

	For $(x, y) \in \CA_2 \cap \CB_1 \cap \CI_2$,
	we choose $\gamma < s_2$ and so we integrate in $y$
	\[
	\begin{split}
	J^{\alpha, \sigma, l}_{2,1,2}(z_1,z_2)
	&\aleq \iint_{\CA_2 \cap \CB_1 \cap \CI_2} \frac{|x-z_1|^{s_1-\alpha-n}\, |x-z_2|^{s_2-n-\gamma}}{|x-y|^{n+2s-\gamma}}\, dx\, dy\\
	&\aleq \int_{\CA_2 \cap \CB_1 \cap \CI_2} |x-z_1|^{s_1-\alpha-n}\, |x-z_2|^{s_2-\gamma-n}\, |x-z_1|^{\gamma-2s}\, dx\\
	&\aleq \int_{\CA_2 \cap \CB_1 \cap \CI_2} |x-z_1|^{\gamma-s_2-\alpha-n}\, |x-z_2|^{s_2-\gamma-n}\, dx\\
	&\aleq \delta^{-\alpha-n}
	\end{split}
	\]
	where the last line follows from H\"older's inequality as above, using that $\gamma-s_2-\alpha < 0$ and $s_2-\gamma > 0$.

	For $(x, y) \in \CA_2 \cap \CB_1 \cap \CI_3$,
	we have 
	$|x-z_1| \aeq |x-z_2|$ and $|x-z_1| \ageq \delta$ and so
	\[
	\begin{split}
	J^{\alpha, \sigma, l}_{2,1,3}(z_1,z_2)
	\aleq& \int_{\CA_2 \cap \CB_1 \cap \CI_3} |x-z_1|^{s_1-\alpha-t-n}\, |x-z_2|^{-s_1+t-n} \,dx\\
	\aeq& \int_{|x-z_1| \ageq \delta} |x-z_1|^{-\alpha-2n} \,dx\\
	\aeq& \ \delta^{-\alpha-n}.
	\end{split}
	\]

	\underline{Estimating $J^{\alpha, \sigma, l}_{2,2,1}$, $J^{\alpha, \sigma, l}_{2,2,2}$ and $J^{\alpha, \sigma, l}_{2,2,3}$} :
	for $(x, y) \in \CA_2 \cap \CB_2$, 
	\[
	\kappa^{\alpha, \sigma}_1(z_1,z_2) \aleq
	 \frac{|x-z_1|^{s_1-\alpha-n}\, |x-z_2|^{s_2-n}}{|x-y|^{n+2s}}.
	\]
	and since in this case $|y-z_1| \aeq |x-y|$,
	\[
	\kappa^{\alpha, \sigma}_2(z_1,z_2) \aleq
	 \frac{|x-z_1|^{s_1-\alpha-n}|y-z_1|^{-\sigma}\, |x-z_2|^{s_2-n}}{|x-y|^{n+2s-\sigma}}
	 \aeq \frac{|x-z_1|^{s_1-\alpha-n} |x-z_2|^{s_2-n}}{|x-y|^{n+2s}}.
	\]
	Integrating w.r.t. $y$, 
	\[
	\begin{split}
	J^{\alpha, \sigma, l}_{2,2,1}(z_1,z_2) 
	&\aleq
	\int_{\CA_2 \cap \CB_2 \cap \CI_1}  |x-z_1|^{s_1-\alpha-n}\, |x-z_2|^{s_2-n}\, |x-z_2|^{-2s}\,dx\\
	&\aleq
	\int_{\CA_2 \cap \CB_2 \cap \CI_1}  |x-z_1|^{s_1-\alpha-n}\, |x-z_2|^{-s_1-n}\, dx.
	\end{split}
	\]
	Let $p > 1$ be small enough so that $(s_1-\alpha-n)p > -n$ and $(-s_1-n)p' < -n$. Then, 
	by H\"older's inequality, it follows that
	\[
	\begin{split}
	J^{\alpha, \sigma, l}_{2,2,1}(z_1,z_2)  
	&\aleq \left(\int_{|x-z_1| \aleq \delta}  |x-z_1|^{(s_1-\alpha-n)p}\, dx\right)^\frac{1}{p} \, \left(\int_{|x-z_2| \ageq \delta}  |x-z_1|^{(-s_1-n)p'}\, dx\right)^\frac{1}{p'}\\
	&\aleq \left( \delta^{(s_1-\alpha-n)p+n} \right)^\frac{1}{p}\, \left( \delta^{(-s_1-n)p'+n} \right)^\frac{1}{p'}\\
	&\aeq \delta^{-\alpha-n}.
	\end{split}
	\]
	
	We also get the $J^{\alpha, \sigma, l}_{2,2,2}$ estimate by the similar way with $J^{\alpha, \sigma, l}_{2,2,1}$.

	For $(x, y) \in \CA_2 \cap \CB_2 \cap \CI_3$, we follow by the same argument as before and after integrating w.r.t. $y$ we have
	\[
	J^{\alpha, \sigma, l}_{2,2,3}(z_1,z_2)
	\aleq \int_{|x-z_1| \ageq \delta} |x-z_1|^{-\alpha-2n}\,dx 
	\ \aleq \ \delta^{-\alpha-n}.
	\]

	\underline{Estimating $J^{\alpha, \sigma, l}_{2,3,1}$, $J^{\alpha, \sigma, l}_{2,3,2}$ and $J^{\alpha, \sigma, l}_{2,3,3}$} :
	let $(x, y) \in \CA_2 \cap \CB_3$. It follows that
	\[
	\kappa^{\alpha, \sigma}_1(z_1,z_2) \aleq \frac{|x-z_1|^{s_1-\alpha-n}\, |y-z_2|^{s_2-n}}{|x-y|^{n+2s}}
	\]
	and 
	\[
	\kappa^{\alpha, \sigma}_2(z_1,z_2) 
	 \aleq \frac{|x-z_1|^{s_1-\alpha-n}|y-z_1|^{-\sigma}\, |y-z_2|^{s_2-n}}{|x-y|^{n+2s-\sigma}}
	 \aleq \frac{|x-z_1|^{s_1-\alpha-n}\, |y-z_2|^{s_2-n}}{|x-y|^{n+2s}}
	\]
	because $|y-z_1| \aeq |x-y|$.
	
% 	Let $p>1$ be such that $(s_2-n)p > -n$ and $(-2s-n)p' < -n$. 
	Recall that $|x-y| \aeq |x-z_2|$ in our setting, so
	\[
	\begin{split}
	&J^{\alpha, \sigma, l}_{2,3,1}(z_1,z_2) \\
	&\aleq \int_{\CA_2 \cap \CB_3 \cap \CI_1} |x-z_1|^{s_1-\alpha-n} \left( \int_{\CA_2 \cap \CB_3 \cap \CI_1}|y-z_2|^{s_2-n}\, |x-y|^{-2s-n}\, dy \right)  \, dx\\
	&\aeq \int_{\CA_2 \cap \CB_3 \cap \CI_1} |x-z_1|^{s_1-\alpha-n} |x-z_2|^{-2s-n} \int_{|y-z_2| \aleq |x-z_2|}|y-z_2|^{s_2-n} dy  \, dx\\
	&\aleq \int_{\CA_2 \cap \CB_3 \cap \CI_1} |x-z_1|^{s_1-\alpha-n}\, |x-z_2|^{-s_1-n}  \,dx.
	\end{split}
	\]
	By H\"older's inequality, let $q>1$ be small enough so that $(s_1-\alpha-n)q > -n$ and $(-s_1-n)q' < -n$ and then,
	\[
	\begin{split}
	J^{\alpha, \sigma, l}_{2,3,1}(z_1,z_2) 
	&\aleq \left( \int_{|x-z_1| \aleq \delta} |x-z_1|^{(s_1-\alpha-n)q}  \,dx \right)^{\frac{1}{q}} \, \left( \int_{|x-z_2| \ageq \delta} |x-z_2|^{(-s_1-n)q'}  \,dx \right)^{\frac{1}{q'}}\\
	&\aleq \left( \delta^{(s_1-\alpha-n)q + n} \right)^{\frac{1}{q}} \, \left( \delta^{(-s_1-n)q' + n} \right)^{\frac{1}{q'}}\\
	&\aeq \delta^{-\alpha-n}.
	\end{split} 
	\]

	To estimate $J^{\alpha, \sigma, l}_{2,3,2}$, we do the same argument as  $J^{\alpha, \sigma, l}_{2,3,1}(z_1,z_2)$:
	let $p > 1$ be such that  $(s_2-n)p > -n$ and $(-2s-n)p' < -n$.  
	By H\"older's inequality,
	\[ 
	\begin{split}
	&J^{\alpha, \sigma, l}_{2,3,2}(z_1,z_2)\\
	&\aleq 
	\iint_{\CA_2 \cap \CB_3 \cap \CI_2} \frac{|x-z_1|^{s_1-\alpha-n}\, |y-z_2|^{s_2-n}}{|x-y|^{n+2s}}\, dx\, dy\\
	&\aleq \int_{\CA_2 \cap \CB_3 \cap \CI_2} |x-z_1|^{s_1-\alpha-n} \left( \int_{\CA_2 \cap \CB_3 \cap \CI_2}|y-z_2|^{s_2-n}\, |x-y|^{-2s-n}\, dy \right)  \,dx\\
	& \aleq \int_{\CA_2 \cap \CB_3 \cap \CI_2} |x-z_1|^{s_1-\alpha-n} \left( \int_{\CA_2 \cap \CB_3 \cap \CI_2}|y-z_2|^{(s_2-n)p}\, dy \right)^{\frac{1}{p}} \left( \int_{\CA_2 \cap \CB_3 \cap \CI_2}|x-y|^{(-2s-n)p'}\, dy \right)^{\frac{1}{p'}} \,dx\\
	&\aleq \int_{\CA_2 \cap \CB_3 \cap \CI_2} |x-z_1|^{s_1-\alpha-n} \left(|x-z_2|^{(s_2-n)p+n}\right)^{\frac{1}{p}}\, \left(|x-z_1|^{(-2s-n)p'+n}\right)^{\frac{1}{p'}}\,  \,dx\\
	&\aeq \int_{\CA_2 \cap \CB_3 \cap \CI_2} |x-z_1|^{-s_2-\alpha-2n+\frac{n}{p'}} |x-z_2|^{s_2-n+\frac{n}{p}}  \,dx.
	\end{split}
	\]

	Again, by H\"older's inequality, let $q>1$ be small enough so that $(-s_2-\alpha-2n+\frac{n}{p'})q < -n$ and $(s_2-n+\frac{n}{p})q' > -n$ and then,
	\[
	\begin{split}
	J^{\alpha, \sigma, l}_{2,3,2}(z_1,z_2)
	&\aleq \left( \int_{|x-z_1| \ageq \delta} |x-z_1|^{(-s_2-\alpha-2n+\frac{n}{p'})q}  \,dx \right)^{\frac{1}{q}} \, \left( \int_{|x-z_2| \aleq \delta} |x-z_2|^{(s_2-n+\frac{n}{p})q'}  \,dx \right)^{\frac{1}{q'}}\\
	&\aleq \left( \delta^{(-s_2-\alpha-2n+\frac{n}{p'})q + n} \right)^{\frac{1}{q}} \, \left( \delta^{(s_2-n+\frac{n}{p})q' + n} \right)^{\frac{1}{q'}}\\
	&\aeq \delta^{-\alpha-n}.
	\end{split} 
	\]

	For $(x, y) \in \CA_2 \cap \CB_3 \cap \CI_3$, we also argue similarly, but since $|x-z_1| \aeq |x-z_2|$ we have
	\[ 
	\begin{split}
	J^{\alpha, \sigma, l}_{2,3,3}(z_1,z_2)
	&\aeq \int _{\CA_2 \cap \CB_3 \cap \CI_3} |x-z_1|^{s_1-\alpha-n} |x-z_2|^{-s_1-n}  \,dx\\
	&\aeq \int_{|x-z_1| \ageq \delta} |x-z_1|^{-\alpha-2n} \,dx\\
	&\aleq \delta^{-\alpha-n}.
	\end{split}
	\]

%-------------------------------- A2
	
	\underline{Estimating $J^{\alpha, \sigma, 2}_{3,j,k}$, for $j = 1,2,3$ and $k = 1,2,3$} :
	the estimate 
	$$
	J^{\alpha, \sigma, 2}_{3,j,k}(z_1,z_2) \aleq \delta^{-\alpha-n}
	$$
	holds for all $j = 1,2,3$ and all $k = 1,2,3$ by using the same proof as $J^{\alpha, \sigma, 2}_{2,j,k}$ for all $j = 1,2,3$ and all $k = 1,2,3$, respectively, since we observe that for $(x, y) \in \CA_3$,
	\[
	\min\{|x-z_1|^{s_1-\alpha-\sigma-n},|y-z_1|^{s_1-\alpha-\sigma-n}\} = |x-z_1|^{s_1-\alpha-\sigma-n}
	\]
	and $|x-z_1| \aeq |x-y|$. 
	
	\underline{Estimating $J^{\alpha, \sigma, 1}_{3,1,1}$, $J^{\alpha, \sigma, 1}_{3,1,2}$ and $J^{\alpha, \sigma, 1}_{3,1,3}$} :
	for $(x, y) \in \CB_1$, by \Cref{la:fundthm}, we have
	\[
	\abs{|x-z_2|^{s_2-n}-|y-z_2|^{s_2-n}} \aleq |x-z_2|^{s_2-n-1}|x-y|.
	\]
	For $(x,y) \in \CA_3$ we have $|x-y| \aeq |x-z_1|$ and thus for $(x,y) \in \CA_3 \cap \CB_1$,
	\[
	\kappa^{\alpha, \sigma}_1(z_1,z_2) \aleq \frac{|y-z_1|^{s_1-\alpha-n}\, |x-z_2|^{s_2-1-n}}{|x-y|^{n+2s-1}} \aeq \frac{|y-z_1|^{s_1-\alpha-n}\, |x-z_2|^{s_2-1-n}}{|x-z_1|^{n+2s-1}}. 
	\]
	Thus,
	\[
	\begin{split}
	J^{\alpha, \sigma, 1}_{3,1,1}(z_1,z_2)
	&\aleq \iint_{\CA_3 \cap \CB_1 \cap \CI_1} \frac{|y-z_1|^{s_1-\alpha-n}\, |x-z_2|^{s_2-1-n}}{|x-y|^{n+2s-1}}\, dx\, dy\\
	&\aeq \int_{\CA_3 \cap \CB_1 \cap \CI_1} \left(\int_{\CA_3 \cap \CB_1 \cap \CI_1} |y-z_1|^{s_1-\alpha-n}\,dy\right) \frac{|x-z_2|^{s_2-1-n}}{|x-z_1|^{n+2s-1}}\, dx\\
	&\aleq \int_{\CA_3 \cap \CB_1 \cap \CI_1}  \frac{|x-z_1|^{s_1-\alpha}\,|x-z_2|^{s_2-1-n}}{|x-z_1|^{n+2s-1}}\, dx\\
	&\aleq \int_{\CA_3 \cap \CB_1 \cap \CI_1}  |x-z_1|^{-s_2-\alpha+1-n}\,|x-z_2|^{s_2-1-n}\, dx.
	\end{split}
	\]
	By H\"older's inequality, let $p>1$ be small enough so that $(-s_2-\alpha+1-n)p > -n$ and $(s_2-1-n)p' < -n$. We obtain
	\[
	\begin{split}
	J^{\alpha, \sigma, 1}_{3,1,1}(z_1,z_2)
	&\aleq \left(  \int_{|x-z_1|\aleq\delta}  |x-z_1|^{(-s_2-\alpha+1-n)p}\, dx \right)^{\frac{1}{p}} \,
	\left(  \int_{|x-z_2|\ageq\delta}  |x-z_2|^{(s_2-1-n)p'}\, dx \right)^{\frac{1}{p'}}\\
	&\aleq \left( \delta^{(-s_2-\alpha+1-n)p+n} \right)^{\frac{1}{p}} \, \left( \delta^{(s_2-1-n)p'+n} \right)^{\frac{1}{p'}}\\
	&\aeq \delta^{-\alpha-n}.
	\end{split}
	\]

	Next, for $(x, y) \in \CA_3 \cap \CB_1 \cap \CI_2$, we consider 	
	\[
	\kappa^{\alpha, \sigma}_1(z_1,z_2) \aleq \frac{|y-z_1|^{s_1-\alpha-n}\, |x-z_2|^{s_2-n}}{|x-y|^{n+2s}} \aeq \frac{|y-z_1|^{s_1-\alpha-n}\, |x-z_2|^{s_2-n}}{|x-z_1|^{n+2s}}. 
	\]
	Observe that $|y-z_1| \aleq |x-z_1|$ and thus $\int_{y \in \CA_3 \cap \CB_1} |y-z_1|^{s_1-\alpha-n}\,dy \aleq |x-z_1|^{s_1-\alpha}$. It follows that 
	\[
	\begin{split}
	J^{\alpha, \sigma, 1}_{3,1,2}(z_1,z_2)
	&\aleq \int_{\CA_3 \cap \CB_1 \cap \CI_2}  \frac{|x-z_1|^{s_1-\alpha}\,|x-z_2|^{s_2-n}}{|x-z_1|^{n+2s}}\, dx\\
	&\aleq \int_{\CA_3 \cap \CB_1 \cap \CI_2}  |x-z_1|^{-s_2-\alpha-n}\,|x-z_2|^{s_2-n}\, dx.
	\end{split}
	\]
	By H\"older's inequality, let $p>1$ be small enough so that $(-s_2-\alpha-n)p < -n$ and $(s_2-n)p' > -n$. We obtain
	\[
	\begin{split}
	&J^{\alpha, \sigma, 1}_{3,1,2}(z_1,z_2)\\
	&\aleq \left(  \int_{|x-z_1|\ageq\delta}  |x-z_1|^{(-s_2-\alpha-n)p}\, dx \right)^{\frac{1}{p}} \,
	\left(  \int_{|x-z_2|\aleq\delta}  |x-z_2|^{(s_2-n)p'}\, dx \right)^{\frac{1}{p'}}\\
	&\aleq \left( \delta^{(-s_2-\alpha-n)p+n} \right)^{\frac{1}{p}} \, \left( \delta^{(s_2-n)p'+n} \right)^{\frac{1}{p'}}\\
	&\aeq \delta^{-\alpha-n}.
	\end{split}
	\]

	For $(x, y) \in \CA_3 \cap \CB_1 \cap \CI_3$,
	\[
	\begin{split}
	J^{\alpha, \sigma, 1}_{3,1,3}(z_1,z_2)
	&\aleq \int_{\CA_3 \cap \CB_1 \cap \CI_3}   |x-z_1|^{-s_2-\alpha+1-n} |x-z_2|^{s_2-1-n} \, dx\\
	&\aeq \int_{|x-z_1| \ageq \delta}  |x-z_1|^{-\alpha-2n}\, dx\\
	&\aleq \delta^{-\alpha-n}.
	\end{split}
	\]

	\underline{Estimating $J^{\alpha, \sigma, 1}_{3,2,1}$, $J^{\alpha, \sigma, 1}_{3,2,2}$ and $J^{\alpha, \sigma, 1}_{3,2,3}$} :
	for $(x, y) \in \CA_3 \cap \CB_2$, we have
	\[
	\kappa^{\alpha, \sigma}_1(z_1,z_2) \aleq \frac{|y-z_1|^{s_1-\alpha-n}\, |x-z_2|^{s_2-n}}{|x-y|^{n+2s}}.
	\]
	
	Thus,
	\[
	\begin{split}
	J^{\alpha, \sigma, 1}_{3,2,1}(z_1,z_2) 
	&\aleq
	\iint_{\CA_3 \cap \CB_2 \cap \CI_1} \frac{|y-z_1|^{s_1-\alpha-n}\, |x-z_2|^{s_2-n}}{|x-y|^{n+2s}}\,dx\, dy\\
	&\aeq \int_{\CA_3 \cap \CB_2 \cap \CI_1} |x-z_2|^{s_2-n}
	\left( \int_{\CA_3 \cap \CB_2 \cap \CI_1} |y-z_1|^{s_1-\alpha-n}\, |x-y|^{-n-2s}\, dy \right)
	\, dx.
	\end{split}
	\]
	Let $p>1$ be so small such that $(s_1-\alpha-n)p > -n$ and $(-2s-n)p' < -n$. Then, 
	by H\"older's inequality, 
	\[
	\begin{split}
	&J^{\alpha, \sigma, 1}_{3,2,1}(z_1,z_2) \\ 
	&\aleq \int_{\CA_3 \cap \CB_2 \cap \CI_1} |x-z_2|^{s_2-n}
	 \left(\int_{\CA_3 \cap \CB_2 \cap \CI_1}  |y-z_1|^{(s_1-\alpha-n)p}\, dy\right)^\frac{1}{p} \left(\int_{\CA_3 \cap \CB_2 \cap \CI_1}  |x-y|^{(-2s-n)p'}\, dy\right)^\frac{1}{p'}\, dx\\
	&\aleq \int_{\CA_3 \cap \CB_2 \cap \CI_1} |x-z_2|^{s_2-n}
	\left(  |x-z_1|^{(s_1-\alpha-n)p+n} \right)^\frac{1}{p} \left(  |x-z_2|^{(-2s-n)p'+n}\, \right)^\frac{1}{p'}\, dx\\
	&\aeq \int_{\CA_3 \cap \CB_2 \cap \CI_1} |x-z_2|^{-s_1-2n+\frac{n}{p'}}
	  |x-z_1|^{s_1-\alpha-n+\frac{n}{p}}\, dx.
	\end{split}
	\]
	Again, by H\"older's inequality, let $q>1$ be small so that $(s_1-\alpha-n+\frac{n}{p})q > -n$ and $(-s_1-2n+\frac{n}{p'})q' < -n$. Hence, 
	\[
	\begin{split}
	&J^{\alpha, \sigma, 1}_{3,2,1}(z_1,z_2) \\ 
	&\aleq \left(\int_{|x-z_1| \aleq \delta} |x-z_1|^{(s_1-\alpha-n+\frac{n}{p})q}\, dx \right)^\frac{1}{q}\, \left(\int_{|x-z_2| \ageq \delta} |x-z_2|^{(-s_1-2n+\frac{n}{p'})q'}\, dx \right)^\frac{1}{q'}\\
	&\aleq \left( \delta^{(s_1-\alpha-n+\frac{n}{p})q+n} \right)^\frac{1}{q}\, \left(\delta^{(-s_1-2n+\frac{n}{p'})q'+n} \right)^\frac{1}{q'}\\
	&\aeq \delta^{-\alpha-n}.
	\end{split}
	\]

	Next, we consider  
	\[ 
	\begin{split}
	J^{\alpha, \sigma, 1}_{3,2,2}(z_1,z_2)
	&\aleq \iint_{\CA_3 \cap \CB_2 \cap \CI_2} \frac{|y-z_1|^{s_1-\alpha-n}\, |x-z_2|^{s_2-n}}{|x-y|^{n+2s}}\, dx\, dy\\
	&\aleq \int_{\CA_3 \cap \CB_2 \cap \CI_2} |x-z_2|^{s_2-n} \left( \int_{\CA_3 \cap \CB_2 \cap \CI_2}|y-z_1|^{s_1-\alpha-n}\, |x-y|^{-n-2s}\, dy \right)\, dx.
	\end{split}
	\]
	
	By H\"older's inequality: we can choose $p > 1$ such that $(s_1-\alpha-n)p > -n$ and $(-2s-n)p' < -n$
	\[
	\begin{split}
	&J^{\alpha, \sigma, 1}_{3,2,2}(z_1,z_2)\\
	&\aleq \int_{\CA_3 \cap \CB_2 \cap \CI_2} |x-z_2|^{s_2-n} \left( \int_{|y-z_1| \aleq |x-z_1|}|y-z_1|^{(s_1-\alpha-n)p} \, dy \right)^{\frac{1}{p}} \left( \int_{|x-y| \ageq |x-z_1|}|x-y|^{(-2s-n)p'}\, dy \right)^{\frac{1}{p'}}  \, dx\\
	&\aleq \int_{\CA_3 \cap \CB_2 \cap \CI_2} |x-z_2|^{s_2-n} \left(|x-z_1|^{(s_1-\alpha-n)p+n}\right)^{\frac{1}{p}}\, \left(|x-z_1|^{(-2s-n)p'+n}\right)^{\frac{1}{p'}}\, dx\\
	&\aleq \int_{\CA_3 \cap \CB_2 \cap \CI_2} |x-z_2|^{s_2-n} |x-z_1|^{-s_2-\alpha-n}\, dx. 
	\end{split}
	\]

	Again, by H\"older's inequality, let $q>1$ be small enough so that $(-s_2-\alpha-n)q < -n$ and $(s_2-n)q' > -n$ we obtain
	\[
	\begin{split}
	J^{\alpha, \sigma, 1}_{3,2,2}(z_1,z_2)
	&\aleq
	\left( \int_{|x-z_1| \ageq \delta} |x-z_1|^{(-s_2-\alpha-n)q}\,dx \right)^{\frac{1}{q}}
	\left( \int_{|x-z_2| \aleq \delta} |x-z_2|^{(s_2-n)q'}\,dx \right)^{\frac{1}{q'}}\\
	&\aleq \left(\delta^{(-s_2-\alpha-n)q +n}\right)^{\frac{1}{q}} \, \left(\delta^{(s_2-n)q' +n}\right)^{\frac{1}{q'}}\\
	&\aeq \delta^{-\alpha-n}.
	\end{split}
	\]

	Similarly,  
	for $(x, y) \in \CA_3 \cap \CB_2 \cap \CI_3$ we get
	\[
	\begin{split}
	J^{\alpha, \sigma, 1}_{3,2,3}(z_1,z_2)
	&\aleq \int_{\CA_3 \cap \CB_2 \cap \CI_3} |x-z_2|^{s_2-n}\, |x-z_1|^{-s_2-\alpha-n}\, \chi_{\CI_3} \,dx\\
	&\aeq  \int_{|x-z_1|\ageq \delta} |x-z_1|^{-\alpha-2n} \,dx\\
	&\aleq \delta^{-\alpha-n}.
	\end{split}
	\]

	\underline{Estimating $J^{\alpha, \sigma, 1}_{3,3,1}$, $J^{\alpha, \sigma, 1}_{3,3,2}$ and $J^{\alpha, \sigma, 1}_{3,3,3}$} : 
	let $(x, y) \in \CA_3 \cap \CB_3$. Then $|x-y| \aeq |x-z_1|$ and $|x-y| \aeq |x-z_2|$, that is \[|x-z_1| \aeq |x-y| \aeq |x-z_2|.\]
	
	If moreover, $(x,y) \in \CI_1$, then $|x-z_1| \aleq \delta$ and $|x-z_2| \ageq \delta$, i.e. from the above, $|x-z_1| \aeq |x-z_2| \aeq \delta$.
	Then $|y-z_1| \aleq |x-z_1| \aeq \delta$ and $|y-z_2| \aleq |x-z_2| \aeq \delta$. So we have for all $(x, y) \in \CA_3 \cap \CB_3 \cap \CI_1$ that 
	\[
	 |y-z_2|, |y-z_1| \aleq \delta
	\]
    However, $|z_1-z_2| = \delta$. Then either $|y-z_1| \aeq \delta$ or $|y-z_2| \aeq \delta$. Indeed, otherwise if $|y-z_1|, |y-z_2| \ll \delta$ then $\delta = |z_1-z_2| \leq |y-z_1| + |y-z_2| \ll 2\delta$, a contradiction.
    
    So for $(x,y) \in \CA_3\cap \CB_3\cap\CI_1$ we have
	\[
	\kappa^{\alpha, \sigma}_1(z_1,z_2) 
	\aleq \frac{|y-z_1|^{s_1-\alpha-n}\, |y-z_2|^{s_2-n}}{|x-y|^{n+2s}}
	\aleq  \delta^{s_1-\alpha-n} \frac{|y-z_2|^{s_2-n}}{|x-y|^{n+2s}} +  
	\delta^{s_2-n} \frac{|y-z_1|^{s_1-\alpha-n}}{|x-y|^{n+2s}}
	\]
 so that 
 \[
 \begin{split}
  J^{\alpha, \sigma, 1}_{3,3,1}(z_1,z_2) \aleq&
  \int_{|y-z_2| \aleq \delta} \int_{|x-z_2| \ageq \delta} \delta^{s_1-\alpha-n} \frac{|y-z_2|^{s_2-n}}{|x-z_2|^{n+2s}}\, dx\,dy\\
  &+\int_{|y-z_1| \aleq \delta} \int_{|x-z_2| \ageq \delta} \delta^{s_1-\alpha-n} \frac{|y-z_1|^{s_2-n}}{|x-z_2|^{n+2s}}\, dx\,dy\\
  \aleq& \ \delta^{-n-\alpha}.
  \end{split}
 \]
The same happens if $(x,y) \in \CI_2$.
 
Finally, we estimate $J^{\alpha, \sigma, 1}_{3,3,3}(z_1,z_2)$. 
	In this case we divide domain into three subcases depending on $y$ as follows:
	$\CJ_i$,	
	\[
	\begin{split}
	\CJ_1 \equiv\CJ_1(z_1,z_2) &= \{ (x, y) \in \R^n \times \R^n : x \in \R^n,|y-z_1| \leq 10 \delta \mbox{ and } |y-z_2| \geq \frac{1}{10}\delta \},\\
	\CJ_2\equiv\CJ_2(z_1,z_2) &= \{ (x, y) \in \R^n \times \R^n : x \in \R^n, |y-z_2| \leq 10\delta \mbox{ and } |y-z_1| \geq \frac{1}{10}\delta \},\\
	\CJ_3\equiv\CJ_3(z_1,z_2) &= \left \{ (x, y) \in \R^n \times \R^n : x \in \R^n, \frac{1}{100} |y-z_2| \leq |y-z_1| \leq 100|y-z_2| \mbox{ and } |y-z_1| \geq \frac{1}{100}\delta \right \}.
	\end{split}
	\]	

	Let $(x, y) \in \CA_3 \cap \CB_3 \cap \CI_3 \cap \CJ_1$. Since $|x-y| \aeq |x-z_2|$, 
	\[
	\kappa^{\alpha, \sigma}_1(z_1,z_2) 
	\aleq \frac{|y-z_1|^{s_1-\alpha-n}\, |y-z_2|^{s_2-n}}{|x-y|^{n+2s}}
	\aeq \frac{|y-z_1|^{s_1-\alpha-n}\, |y-z_2|^{s_2-n}}{|x-z_2|^{n+2s}}.
	\]
	Integrating w.r.t. $x$,
	\[
	\begin{split}
	J^{\alpha, \sigma, 1}_{3,3,3}(z_1,z_2)
	&\aleq \int_{\CA_3 \cap \CB_3 \cap \CI_3 \cap \CJ_1}  \int_{|x-z_2|\ageq|y-z_2|}
	\frac{|y-z_1|^{s_1-\alpha-n}\,|y-z_2|^{s_2-n}}{|x-z_2|^{n+2s}} \,dx\,dy\\
	&\aleq \int_{\CA_3 \cap \CB_3 \cap \CI_3 \cap \CJ_1} 
	 |y-z_1|^{s_1-\alpha-n}\,|y-z_2|^{s_2-n}\,|y-z_2|^{-2s} \,dy\\
	&\aeq \int_{\CA_3 \cap \CB_3 \cap \CI_3 \cap \CJ_1} 
	 |y-z_1|^{s_1-\alpha-n}\,|y-z_2|^{-s_1-n} \,dy.
	\end{split}
	\]	
	Let $p>1$ be small enough so that $(s_1-\alpha-n)p < -n$ and $(-s_1-n)p' > -n$. Then, by H\"older's inequality, we obtain that
	\[
	\begin{split}
	J^{\alpha, \sigma, 1}_{3,3,3}(z_1,z_2)
	&\aleq \left(\int_{|y-z_1| \aleq \delta} |y-z_1|^{(s_1-\alpha-n)p}\, dy\right)^\frac{1}{p}\, \left(\int_{|y-z_2| \ageq \delta} |y-z_2|^{(-s_1-n)p'}\, dy\right)^\frac{1}{p'}\\
	&\aleq \left(\delta^{(s_1-\alpha-n)p+n}\right)^\frac{1}{p}\,\left(\delta^{(-s_1-n)p'+n}\right)^\frac{1}{p'}\\
	&\aeq \delta^{-\alpha-n}.
	\end{split}
	\]
	
	Next, we let $(x, y) \in \CA_3 \cap \CB_3 \cap \CI_3 \cap \CJ_2$. Then
	\[
	\kappa^{\alpha, \sigma}_1(z_1,z_2) 
	\aleq \frac{|y-z_1|^{s_1-\alpha-n}\, |y-z_2|^{s_2-n}}{|x-y|^{n+2s}}
	\aeq \frac{|y-z_1|^{s_1-\alpha-n}\, |y-z_2|^{s_2-n}}{|x-z_1|^{n+2s}}
	\]	
	because $|x-y| \aeq |x-z_1|$. Similarly as above case, we integrate w.r.t. $x$ and use H\"older's inequality for $p>1$ such that $\frac{1}{p}+\frac{1}{p'}=1$, $(-s_2-\alpha-n)p < -n$ and $(s_2-n)p' > -n$. It follows that
	\[
	\begin{split}
	J^{\alpha, \sigma, 1}_{3,3,3}(z_1,z_2)
	&\aleq \int_{\CA_3 \cap \CB_3 \cap \CI_3 \cap \CJ_2}  \int_{|x-z_1|\ageq|y-z_1|}
	\frac{|y-z_1|^{s_1-\alpha-n}\,|y-z_2|^{s_2-n}}{|x-z_1|^{n+2s}} \,dx\,dy\\
	&\aleq \int_{\CA_3 \cap \CB_3 \cap \CI_3 \cap \CJ_2} 
	|y-z_1|^{s_1-\alpha-n}\,|y-z_2|^{s_2-n}\,|y-z_1|^{-2s} \,dy\\
	&\aeq \int_{\CA_3 \cap \CB_3 \cap \CI_3 \cap \CJ_2} 
	|y-z_1|^{-s_2-\alpha-n}\,|y-z_2|^{s_2-n} \,dy\\
	&\aleq \left(\int_{|y-z_1| \ageq \delta} |y-z_1|^{(-s_2-\alpha-n)p}\, dy\right)^\frac{1}{p}\, \left(\int_{|y-z_2| \aleq \delta} |y-z_2|^{(s_2-n)p'}\, dy\right)^\frac{1}{p'}\\
	&\aleq \left(\delta^{(-s_2-\alpha-n)p+n}\right)^\frac{1}{p}\,
	\left(\delta^{(s_2-n)p'+n}\right)^\frac{1}{p'}\\
	&\aeq \delta^{-\alpha-n}.
	\end{split}
	\]
	
	For $(x, y) \in \CA_3 \cap \CB_3 \cap \CI_3 \cap \CJ_3$, we have $|y-z_2| \aeq |y-z_1| \ageq \delta$ and so
	\[
	\kappa^{\alpha, \sigma}_1(z_1,z_2) 
	\aleq \frac{|y-z_1|^{s_1-\alpha-n}\, |y-z_2|^{s_2-n}}{|x-y|^{n+2s}}
	\aeq \frac{|y-z_1|^{2s-\alpha-2n}}{|x-z_1|^{n+2s}}.
	\]
	Thus, 
	\[
	\begin{split}
	J^{\alpha, \sigma, 1}_{3,3,3}(z_1,z_2)
	&\aleq \int_{\CA_3 \cap \CB_3 \cap \CI_3 \cap \CJ_3}  \int_{|x-z_1|\ageq|y-z_1|}
	\frac{|y-z_1|^{2s-\alpha-2n}}{|x-z_1|^{n+2s}} \,dx\,dy\\
	&\aleq \int_{|y-z_1| \ageq \delta} 
	|y-z_1|^{2s-\alpha-2n}\,|y-z_1|^{-2s} \,dy\\
	&\aeq \int_{|y-z_1| \ageq \delta} 
	|y-z_1|^{-\alpha-2n} \,dy\\
	&\aeq \delta^{-\alpha-n}.
	\end{split}
	\]
\end{proof}

Let us now prove \Cref{pr:se:1}.
\begin{proof}[Proof of \Cref{pr:se:1}]
	Let $z_1, z_2 \in \R^n$. The first inequality \eqref{eq:CZ:g1} true by \Cref{pr:se:1v2} with $\alpha = 0$. 
	Next, we observe that $A_{K,s_1,s_2}$ may not be symmetric in general (unless $s_1=s_2=s$). However, since for our setup the values of $s_1$ and $s_2$ are interchangeable, \eqref{eq:CZ:g2} and \eqref{eq:CZ:g3} are equivalent. Thus, it suffices to prove \eqref{eq:CZ:g2}. 
	By \Cref{la:mvf}, we have for any $\alpha,\sigma \in [0,1]$,
	\[
	\begin{split}
	&\left | |x-z_1+h|^{s_1-n}-|y-z_1+h|^{s_1-n} - \brac{|x-z_1|^{s-n}-|y-z_1|^{s-n}} \right |\\
	&\aleq |h|^\alpha \brac{\left | |x-z_1+h|^{s-\alpha-n}- |y-z_1+h|^{s-\alpha-n}\right | + \abs{|x-z_1|^{s-\alpha-n}-|y-z_1|^{s-\alpha-n}}}\\
	& +|h|^\alpha \min\{|x-z_1|^{s-\alpha-\sigma-n},|y-z_1|^{s-\alpha-\sigma-n}\}\, |a-b|^\sigma.
	\end{split}
	\]
	We choose $\sigma$ large enough and $\alpha$ small enough so that the assumptions of \Cref{pr:se:1v2} are satisfied.
	
	By assumption, we have $|h| \le \frac{1}{2}|z_1 - z_2|$ and so $|z_1 + h - z_2| \aeq |z_1 - z_2|$. 
	Using \Cref{pr:se:1v2},
	\[
	\begin{split}
	&\left |A_{K,s_1,s_2}(z_1+h,z_2) - A_K(z_1,z_2) \right | \\
	&\aleq |h|^\alpha \abs{\kappa_\alpha(z_1+h, z_2) + \kappa_\alpha(z_1, z_2) + \kappa_\sigma(z_1, z_2) }\\
	&\aleq |h|^\alpha \abs{|z_1 + h - z_2|^{-\alpha-n} + 2|z_1 - z_2|^{-\alpha-n}}\\
	&\aleq |h|^\alpha |z_1 - z_2|^{-\alpha-n}. 
	\end{split}
	\]
	This completes the proof.
\end{proof}

\section{Application to nonlocal PDEs: Proof of Theorem~\ref{th:pdeappl}}\label{s:appl}
\begin{proof}
As shown in \cite{MSY20}, 
	\[
	\mathcal{L}^s_{K} u = g
	\]
	is equivalent to
	\[
	T_{K,s_1,s_2} \laps{s_1} u = c \lapms{s_2} g.
	\]
	Dividing both sides by $\sup K$ we have 
	\[
	T_{\frac{K}{\sup K},s_1,s_2} \laps{s_1} u = c \frac{1}{\sup K}\lapms{s_2} g.
	\]
	and thus
	\[
	\laps{s_1} u + T_{\frac{K}{\sup K}-1,s_1,s_2} \laps{s_1} u = c \frac{1}{\sup K}\lapms{s_2} f.
	\]
	Set $\tilde{K} := \frac{K}{\sup K}-1$, then $\|\tilde{K}\|_{L^\infty} < \eps$. 
	By \Cref{th:cz}, $T_{\tilde{K},s_1,s_2}: L^p(\R^n) \to L^p(\R^n)$ has norm \[
	\|T_{\tilde{K},s_1,s_2}\|_{L^p \to L^p} < 1.                                                  
	\]
	So
	\[
	I - T_{\tilde{K},s_1,s_2}
	\]
	is invertible as an operator from $L^p(\R^n) \to L^p(\R^n)$; indeed the following is a uniformly convergent series
	\[
	(I - T_{\tilde{K},s_1,s_2})^{-1} = \sum_{k=0}^\infty (T_{\tilde{K},s_1,s_2})^k.
	\]
	For $s_1 \in (0,1)$ the solution $\tilde{u} \in \dot{H}^{s_1,p}$ 
	\[
	 \laps{s_1} \tilde{u} = (I - T_{\tilde{K},s_1,s_2})^{-1}\lapms{s_2} g
	\]
    differs from $u$ at most by a constant. So we get the desired estimate taking $g:=\laps{s} f$.
	\end{proof}
\bibliographystyle{abbrv}
\bibliography{bib}

\end{document}